\documentclass[12pt,reqno]{amsart}
\usepackage{hyperref}

\usepackage{amsfonts, amssymb,amsmath,amsthm,enumerate}
\usepackage[margin=1.5in]{geometry}

\numberwithin{equation}{section}

\newtheorem{theorem}{Theorem}[section]
\newtheorem{lemma}[theorem]{Lemma}
\newtheorem{proposition}[theorem]{Proposition}

\theoremstyle{definition}

\newtheorem{remark}[theorem]{Remark}
\newtheorem{definition}[theorem]{Definition}

\renewcommand{\d}{\mathrm{d}}
\renewcommand{\phi}{\varphi}

\newcommand{\0}{\mathbf{0}}

\newcommand{\FF}{\mathbb{F}}
\newcommand{\ZZ}{\mathbb{Z}}

\newcommand{\QQ}{\mathbb{Q}}
\newcommand{\RR}{\mathbb{R}}

\renewcommand{\leq}{\leqslant}

\renewcommand{\geq}{\geqslant}

\newcommand{\x}{\mathbf{x}}
\newcommand{\y}{\mathbf{y}}

\renewcommand{\ss}{\mathfrak{S}}

\newcommand{\al}{\alpha}

\DeclareMathOperator{\rank}{rank}

\DeclareMathOperator{\Mod}{mod} 

\renewcommand{\bmod}[1]{\,(\Mod{#1})}

\newcommand{\minor}{\mathfrak{m}}
\newcommand{\major}{\mathfrak{M}}

\newcommand{\eeq}{\end{equation}}
\newcommand{\beql}[1]{\begin{equation}\label{#1}}

\newcommand{\Z}{\mathbb{Z}}
\newcommand{\Q}{\mathbb{Q}}
\newcommand{\R}{\mathbb{R}}
\newcommand{\C}{\mathbb{C}}
\newcommand{\T}{\mathbb{T}}
\newcommand{\F}{\mathbb{F}}
\newcommand{\N}{\mathbb{N}}
\newcommand{\A}{\mathbb{A}}

\newcommand{\Sing}{\mathrm{Sing}}
\newcommand{\intd}{\mathrm{d}}
\newcommand{\hash}{\#}
\newcommand{\supp}{\mathrm{supp}}
\newcommand{\eps}{\varepsilon}

 \newcommand{\set}[1]{\left\{#1\right\}}
\newcommand{\bigset}[1]{\bigl\{ #1 \bigr\}}
\newcommand{\Bigset}[1]{\Bigl\{ #1 \Bigr\}}

\newcommand{\abs}[1]{\left| #1\right|}
\newcommand{\bigabs}[1]{\bigl| #1 \bigr|}
\newcommand{\Bigabs}[1]{\Bigl| #1 \Bigr|}

\newcommand{\floor}[1]{\left\lfloor #1 \right\rfloor}
\newcommand{\brac}[1]{\left( #1 \right)}
\newcommand{\bigbrac}[1]{\bigl( #1 \bigr)}
\newcommand{\Bigbrac}[1]{\Bigl( #1 \Bigr)}

\newcommand{\norm}[1]{\left\| #1\right\|}
\newcommand{\bignorm}[1]{\big\| #1 \big\|}

\newcommand{\recip}[1]{\frac{1}{#1}}
\newcommand{\trecip}[1]{\tfrac{1}{#1}}
\newcommand{\ve}{\mathbf{e}}
\newcommand{\vx}{\mathbf{x}}
\newcommand{\vy}{\mathbf{y}}

\newcommand{\vmu}{\boldsymbol{\mu}}

\newcommand{\vlambda}{\boldsymbol{\lambda}}

\newcommand{\va}{\mathbf{a}}

\newcommand{\vh}{\mathbf{h}}

\newcommand{\vc}{\mathbf{c}}

\newcommand{\vt}{\mathbf{t}}

\newcommand{\vm}{\mathbf{m}}

\newcommand{\vr}{\mathbf{r}}

\newcommand{\vs}{\mathbf{s}}

\newcommand{\vepsilon}{\boldsymbol{\epsilon}}

\newcommand{\M}{\mathfrak{M}}
\newcommand{\m}{\mathfrak{m}}

\begin{document}

\title[Improvements in Birch's theorem]{Improvements in Birch's theorem on\\ forms in many variables}

\author{T.D.~Browning}
\address{School of Mathematics\\
University of Bristol\\ Bristol\\ BS8 1TW}
\email{t.d.browning@bristol.ac.uk}
\author{S.M.~Prendiville}
\address{Department of Mathematics and Statistics\\University of Reading\\PO Box 220\\ Reading\\RG6 6AX}
\email{s.m.prendiville@reading.ac.uk}


\begin{abstract}
We show that a non-singular integral form of degree $d$ is soluble over the integers if and only if it is soluble over $\RR$ and over $\QQ_p$ for all primes $p$, provided that the form
has at least $(d-\frac{1}{2}\sqrt{d})2^d$ variables. 
This improves on a longstanding result of Birch. 
\end{abstract}

\maketitle
\setcounter{tocdepth}{1}
\tableofcontents

\thispagestyle{empty}

\section{Introduction}

Let $F\in \ZZ[x_1,\dots,x_n]$
be 
a homogeneous polynomial  of degree $d\geq 3$.
A fundamental ambition in number theory is to determine when the Diophantine equation 
\begin{equation}\label{eq}
F(x_1,\dots,x_n)=0
\end{equation}
has a non-trivial integral solution.
The Hardy--Littlewood circle method has been extraordinarily effective at answering this question for typical $F$ when the number of variables is sufficiently large in terms of $d$.
An obvious necessary condition for the solubility of \eqref{eq} in integers is that it should be {\em everywhere locally soluble}, by which we mean that it has non-trivial solutions over $\RR$ and $\QQ_p$ for every prime $p$.
According to a  renowned result of Birch \cite{birch},  these conditions are sufficient provided that $F$ is non-singular  and 
$
n>(d-1)2^d.
$
It is possible to relax the non-singularity condition  by 
imposing stronger constraints on $n$ and local solubility.  For the latter, 
Birch  asks instead for the system
\begin{equation}\label{eq'}
F(x_1,\dots,x_n)=0, \quad \nabla F(x_1,\dots,x_n)\neq \0
\end{equation}
to be everywhere locally soluble. We say that $F$ satisfies the {\em smooth Hasse principle} if this condition is sufficient to ensure that this system also has a non-trivial  integral solution.
Allowing  $\sigma$ to denote the (affine) dimension of the  {\em singular locus}
 cut out by the system of equations 
$\nabla F(x_1,\dots,x_n)=\0$, it follows from Birch's investigation  \cite{birch} that 
$F$ satisfies the smooth Hasse principle provided that 
\begin{equation}\label{eq:birch}
n-\sigma>(d-1)2^d.
\end{equation}
Note that $\sigma\in \{0,\dots,n-1\}$, with $\sigma=0$ if and only if $F$ is  non-singular.

Birch's theorem has had an extensive impact on number theory, with the underlying 
tools being adapted to handle numerous problems. This includes, but is not limited to:
\begin{itemize}
\item[---]
the vanishing of $F$ on general  $\ZZ$-linear subspaces  (Brandes \cite{brandes});
\item[---]
a generalisation to the  function field $\FF_q[t]$  (Lee \cite{TW-Lee}); 
\item[---]
a generalisation to bihomogeneous forms (Schindler \cite{damaris});
\item[---]
a generalisation to arbitrary number fields (Skinner \cite{skinner}).
\end{itemize}
Activity around 
reducing the lower bound  \eqref{eq:birch} 
for $n-\sigma$
in Birch's original result, however,  has not been so 
vigorous.

The most impressive improvement to date arises in the case $d=3$ of cubic forms. Thus, it follows from work of Hooley \cite{hooley-88} 
that the smooth Hasse principle holds for integral cubic forms provided that $n-\sigma\geq 9$. 
Moreover, Heath-Brown \cite{14} has shown that  any integral cubic form has a non-trivial integer zero provided that $n\geq 14$, with no restriction on the singular locus,  the question of local solubility being automatic. 
The only other improvement to date pertains to the case $d=4$. In this setting, Browning and Heath-Brown \cite{41} have established the smooth Hasse principle for integral quartic forms provided that $n-\sigma\geq 41$, saving $8$ variables over the approach taken by Birch.  Finally, 
this inequality has been sharpened to $n-\sigma\geq 40$ by Hanselmann \cite{markus}.

Our main result 
improves on \eqref{eq:birch} for  every degree.

\begin{theorem}\label{general d improvement theorem}  
Let $F \in \Z[x_1, \dots, x_n]$ be a form of degree $d\geq 3$ with singular locus of dimension $\sigma$.  
Suppose that
$$
n - \sigma \geq \Bigbrac{d - \trecip{2} \sqrt{d}} 2^d.
$$
Then the smooth Hasse principle holds for $F$.
\end{theorem}

As we shall see shortly the proof of this result is based on a generalisation of the method in \cite{41}.
One verifies that the admissible range for $n$ is weaker than that provided by \cite[Thm.~1]{41} when $d=4$. 
In fact, for smaller values of $d$ we are able to get a much more significant 
improvement,  as in the following result.

\begin{theorem}\label{small d improvement theorem}  
Let $F \in \Z[x_1, \dots, x_n]$ be a form of degree $d\in \{3,4,5,6,7,8,9\}$ with singular locus of dimension $\sigma$.  
Suppose that
$$
n - \sigma > \tfrac{3}{4} d 2^d - 2d.
$$
Then the smooth Hasse principle holds for $F$.
\end{theorem}

The question of determining when the system 
\eqref{eq'} is everywhere locally soluble is far from being decided. 
Denoting by  $\nu_d(p)$  the least integer $n$ such that every degree $d$
form  $F \in \Z[x_1, \dots, x_n]$ has a zero in $\QQ_p$,  Artin conjectured that $\nu_d(p)=d^2+1$ for every prime $p$ (this is known to be false for even $d$ but is still open for forms of odd degree).  
Specialising to the case $d=5$ of quintic forms, where solubility over $\RR$ is automatic, 
it was shown by
Leep and Yeoman \cite{leep} that $\nu_5(p)=26$ for $p\geq 47$. This was strengthened by Heath-Brown
\cite{hb-padic}, so that this
 equality holds for 
$p\geq 17$. 
In particular, when  Theorem \ref{small d improvement theorem} is applied to non-singular quintic forms,  it suffices to check  the solubility over $\QQ_p$ for primes $p\leq 13.$ In this range, the best result  we have is due to 
Zahid \cite{zahid}, who establishes that 
$\nu_5(p)\leq 4562912$ for $p\leq 13$.

We shall give an overview of the proof of our main  results in \S\ref{overview}.
Taking  $d=3$ in Theorem \ref{small d improvement theorem}, we obtain the Hasse principle for non-singular cubic forms in at least $13$ variables. 
It is no coincidence that this coincides with the constraint arising in Skinner's work \cite{skinner13} on non-singular cubic forms over number fields. Indeed, when $d=3$ our proof reduces to the argument in \cite{skinner13} (which over $\QQ$ is Heath-Brown's seminal work \cite{hb10} --- without a Kloosterman refinement).
When $d=4$ the inequality in Theorem~\ref{small d improvement theorem} 
recovers the conclusions of \cite{41} precisely. When $d=5$, for example,  we witness a saving of 
18 variables over Birch's result.  

Birch  \cite{birch} has an analogous result for general systems of integral forms $F_1,\dots,F_R$ of equal degree. 
It would be interesting to determine whether the methods of this paper can be developed to produce comparable improvements  for $R>1$.  Similarly, once suitably modified,   it 
is natural to hope that our argument  yields corresponding improvements in  the generalisations \cite{brandes, TW-Lee, damaris, skinner} discussed above.

\subsection*{Acknowledgements}
While working on this paper the authors were 
supported by the Leverhulme Trust and ERC grant \texttt{306457}.

\section{Preliminaries}\label{overview}

Our proof of 
Theorems \ref{general d improvement theorem} 
and~\ref{small d improvement theorem} 
proceeds via the Hardy--Littlewood circle method.  In this section we outline the strategy of the proof, together with some    conventions regarding notation and  some preliminary technical results.

The overall   goal is to establish an asymptotic formula
for the quantity 
\begin{equation}\label{eq:def-N}
N_\omega(F;P):= \sum_{\substack{\x\in \ZZ^n\\ F(\x)=0}} 
\omega(\x/P),  
\end{equation}
as $P\rightarrow \infty$, for a suitable weight function $\omega:\RR^n
\rightarrow [0,\infty)$.  We show that under the assumptions of Theorem \ref{general d improvement theorem} or \ref{small d improvement theorem} on $n-\sigma$, there is a  constant $c_F>0$ such that 
$$
N_\omega(F;P)\sim c_F  P^{n-d},
$$
provided that the system \eqref{eq'} is everywhere locally soluble.

Our  starting point 
is the  identity
$$
N_\omega(F;P)=\int_\T S(\alpha, P) \d\alpha,  
$$
where $\T = \R/\Z$ and
\begin{equation}\label{eq:def-S}
S(\alpha, P) := \sum_{\vx \in \Z^n} \omega(\vx/P) e(\alpha F(\vx)).
\end{equation}
The idea is then to divide the torus $\T$ into a
set of major arcs $\major$ and minor arcs $\minor$. 
Given $\Delta > 0$ we define the major arcs $\M =\M(\Delta)$ to be the set
\begin{equation}\label{major arc def}
\M  =  \bigcup_{q \leq P^\Delta} \bigcup_{\substack{ a \leq q \\ (a,q) = 1}}\set{\alpha \in \T : \bignorm{\alpha- \tfrac{a}{q}} \leq P^{\Delta-d}}.
\end{equation}
These are non-overlapping provided that $\Delta<\frac{d}{3}$.
We define the minor arcs to be their complement 
$
\m = \T\setminus\M$.  
In the usual way we seek to prove an asymptotic formula 
\begin{equation}\label{eq:major}
\int_{\major} S(\alpha,P) \d \alpha\sim c_{F} P^{n-d},
\end{equation}
as $P\rightarrow \infty$, together with a satisfactory bound on the minor arcs
\begin{equation}\label{eq:minor}
\int_{\minor} S(\alpha,P) \d \al  
=o( P^{n-d}).
\end{equation}
Here the constant $c_{F}$ turns out to be a product of local densities which will be positive if 
the system \eqref{eq'} is everywhere locally soluble.
The  treatment of \eqref{eq:major} is standard and is the focus of \S \ref{major arc section}.  

Our main innovation lies in our treatment of \eqref{eq:minor}. The plan is to  develop
extensively 
 the approach adopted in  \cite{41} to estimate $S(\alpha,P)$ when $F$ is a quartic form. 
This  relied on a single application of van der Corput differencing to get a family of exponential sums involving cubic polynomials. These were then estimated directly using Poisson summation, rather than through further differencing operations. 
In our work, which deals with forms of degree $d$,  we produce two key estimates for $S(\alpha,P)$ 
in \S \ref{s:exp}. The first 
(Proposition~\ref{birch + van der corput}) is obtained via 
$d-k$ applications of  van der Corput differencing together 
with an application of Birch's bound from \cite{birch}
(suitably modified), as it applies to 
exponential sums with underlying polynomials of degree at most $k$.
The second result (Proposition \ref{vdc + cubic lemma}) is  proved using  $d-3$  applications of  van der Corput differencing  together with the bound for cubic exponential sums from \cite{41} obtained using   Poisson summation.

The final treatment of \eqref{eq:minor} is carried out in \S \ref{s:minor}.
It is somewhat  disappointing that we are unable to cover all of the  minor arcs when 
$n-\sigma>
\tfrac{3}{4} d 2^d - 2d$ for any $d\geq 3$.
As we shall see in 
Remark \ref{remark}, however,   the criterion that emerges from our 
deliberations requires $n-\sigma$ to be 
 asymptotically $d 2^d$.

\medskip

The remainder of this section is taken up with introducing notation and proving some preliminary technical results. 
Given $\vepsilon \in (\N\cup\{0\})^n$ and a sufficiently differentiable function $g :\R^n \to \C$, put 
$$
\partial^{\vepsilon} g=
\frac{\partial^{\epsilon_1 + \dots + \epsilon_n}g}{\partial x_1^{\epsilon_1} \dotsm \partial x_n^{\epsilon_n}}.
$$
The following result follows from partial summation and induction on the dimension.

\begin{lemma}[Partial summation formula]\label{partial summation lemma}{\ }
Let  $\phi : \set{1, \dots, N}^n \to \C$ be a function and let 
$$
T_\phi(\vt) := \sum_{1 \leq x_1 \leq t_1} \dots \sum_{1 \leq x_n \leq t_n} \phi(\vx).
$$
Then for any $g \in C^n(\R^n)$ we have
\begin{align*}
\sum_{1 \leq \vx \leq N} &g(\vx) \phi(\vx)\\
& = \sum_{\vepsilon \in \set{0,1}^n} \frac{(-N)^{\epsilon_1 + \dots + \epsilon_n} }{N^{n}} \int_{[0,N]^{n}}\partial^{\vepsilon} g\bigbrac{N \overline{\vepsilon} + \vt_{\vepsilon}} T_\phi\bigbrac{N \overline{\vepsilon} + \vt_{\vepsilon}} \intd\vt,
\end{align*}
where 
$\overline{\vepsilon}=(1, 1, \dots, 1) - \vepsilon$ and 
 $\vt_{\vepsilon}$ denotes  the vector whose $i$th coordinate equals zero if $\epsilon_i = 0$ and equals $t_i$ if $\epsilon_i = 1$.  
\end{lemma}

When $\alpha \in \R$ we write $\norm{\alpha}$ for the distance from $\alpha$ to the nearest integer, a function which induces a metric on $\T = \R/\Z$ via $d(\alpha, \beta) = \norm{\alpha - \beta}$.  We use absolute values  $|\vx|$ to denote the norm $\max_i |x_i|$
\begin{lemma}[Shrinking lemma]\label{shrinking lemma}
Given a symmetric $n\times n$ real matrix $A$, define $N_A(H, \lambda)$ to be the number of $\vh \in \Z^n$ satisfying $|\vh| \leq H$ and $\norm{(A\vh)_j} \leq \lambda$ for all $j$.  Then for any $H \geq 1$, $\lambda \in (0, 1/H]$ and $\theta \in (0, 1]$, we have the estimate
$$
N_A(H, \lambda) \ll (H^{-1} + \theta)^{-n} N_A(\theta H, \theta \lambda).
$$
\end{lemma}

\begin{proof}  This result is a consequence of a result by Davenport \cite[Lemma 12.6]{dav}, which is proved using the geometry of numbers.  To be precise 
the statement of \cite[Lemma 12.6]{dav}  gives the bound
$
N_A(H, \lambda) \ll \theta^{-n} N_A(\theta H, \theta \lambda).
$
To see how this implies the lemma, first note that if $\theta \geq H^{-1}$ then 
$
\theta^{-1} \leq 2(H^{-1} + \theta)^{-1}
$ 
and we are done.  Next, if $\theta < H^{-1}$ then $N_A(\theta H, \theta \lambda) = 1$ and 
$
(H^{-1} + \theta)^{-1} > \trecip{2} H.
$  
Therefore the trivial estimate gives
\begin{align*}
(H^{-1} + \theta)^{-n} N_A(\theta H, \theta \lambda) & \geq 2^{-n} H^n
 \gg N_A(H, \lambda),
\end{align*}
as required.
\end{proof}

Our next result involves Diophantine approximation.
 Given $\alpha \in \T$ and $q \in \N$ we say that $\alpha$ and $q$ are \emph{primitive} if there exists $a \in \Z$ with $(a, q) = 1$ such that $\norm{q\alpha} = |q\alpha - a|$.
Notice that if $q$ and $\alpha$ are not primitive then one can find a divisor $q_0$ of $q$ which is primitive to $\alpha$ and which satisfies $\norm{q_0\alpha} \leq \frac{q_0}{q} \norm{q\alpha} < \norm{q\alpha}$.
The following simple result is due to 
Heath-Brown \cite[Lemma~2.3]{14}.

\begin{lemma}\label{heathbrown lemma}  Let $\alpha \in \T$ and $q \in \N$ be primitive.  Suppose that $m \in \Z$ satisfies:
\emph{
\begin{enumerate}[(i)]
\item $|m| < \recip{2} \norm{q\alpha}^{-1}$;
\item $\norm{m\alpha} < \recip{2} q^{-1}$; and 
\item \emph{$|m| < q$ or $\norm{m\alpha} < \norm{q\alpha}$.}
\end{enumerate}}
Then $m = 0$.
\end{lemma}

\begin{proof}
Since $\alpha$ and $q$ are primitive, there exists $a \in \Z$ with $(a, q) = 1$ and $\norm{q\alpha} = |q\alpha - a|$.   Our formulation of the lemma now follows from 
\cite[Lemma~2.3]{14} with $P_0=2q$.
\end{proof}

The remaining results in this section involve viewing various varieties that are defined over $\QQ$ over several different finite fields. 
To simplify the exposition, write $\F_\infty$ for $\Q$.  Given a form $G \in \Z[x_1, \dots, x_n]$ and $\nu$ a prime or the prime at infinity, define the \emph{singular locus} of $G$ over $\F_\nu$ to be the algebraic set 
\begin{equation}\label{singular locus}
\Sing_\nu(G) := \bigset{ \vx \in \A_{\F_\nu}^n: \mbox{$\partial^{\ve_i} G(\vx)  = 0$ for $1 \leq i \leq n$}}.
\end{equation}
Here $\ve_i$ denotes the $i$th standard basis vector.  Throughout we use the notation
\begin{equation}\label{sigma dimension notation}
\sigma_\nu(G) := \dim \Sing_{\nu}(G).
\end{equation}
Denote the positive part of a real number $x$ by 
$$
x^+:= \max\set{x, 0}.
$$  
The following is Lemma 1 of Browning and Heath-Brown \cite{41}.

\begin{lemma}\label{dimension bound lemma}
 Let $G \in \Z[x_1, \dots, x_n]$ be a form of degree $d$ whose singular locus over $\F_\nu$ has dimension $\sigma_\nu(G)$.  Define 
 $$
 B_{\nu}(G,s) := \set{\vh \in \A_{\F_\nu}^n : \sigma_\nu( \vh \cdot \nabla G) \geq s}.
 $$
 Then, provided $\nu$ is coprime to $d$, the set $B_{\nu}(G,s)$ is an affine variety, defined by $O_d(1)$ equations, each of degree $O_d(1)$, with
$$
\dim B_{\nu}(G,s) \leq n - (s - \sigma_\nu(G))^+.
$$
  \end{lemma}

Next, for $\va \in \Z^n$, let us write $[\va]_p$ for the image of $\va$ under the natural projection $\Z^n \to \F_p^n$.  For consistency, we write $[\va]_\infty$ for $\va$.  The following is a simple consequence of Browning and Heath-Brown \cite[Lemma~4]{41}. 
 \begin{lemma}[Dimension growth bound]\label{counting lemma}
 Let $\mathcal{P}\subset \set{p : p \text{ prime}} \cup \set{\infty}$ be a finite subset.  To each $\nu \in \mathcal{P}$ we associate an affine  variety $X_\nu \subset \A_{\F_\nu}^n$ defined by at most $D$ equations with coefficients in $\F_\nu$, each of degree at most $D$.  Suppose that the dimension of $X_\nu$ is at most $k_\nu$.
 Then there exists $A(D, n) > 0$ such that for any $T \geq 1$ we have
\begin{align*}
  \hash \{ \va \in \Z^n \cap [-T, T]^n :  [\va]_\nu \in &X_\nu \text{ for all } \nu \in \mathcal{P}\} \\ &\leq A(D, n)^{|\mathcal{P}|}\ \sum_{\nu \in \mathcal{P}}  T^{k_\nu} \prod_{\substack{\mu \in \mathcal{P}\\ k_\mu < k_\nu}} \mu^{-(k_\nu - k_\mu)},
\end{align*}
 where we interpret $\mu^{-1}$ to be $0$ when $\mu = \infty$.
 \end{lemma}

 \begin{proof}
We describe how to deduce the above from \cite[Lemma~4]{41}.  
Let $N(T)$ denote the cardinality that is to be estimated. 
Define the set 
\begin{equation}\label{primes less than l}
 \set{p_1, \dots, p_r} := \set{\nu \in \mathcal{P} : k_{\nu} < k_{\infty}}.  
 \end{equation}
 Notice that $p_1, \dots, p_r$ are all necessarily finite primes.
 Writing $\kappa_i$ for $k_{p_i}$, let us order the $p_i$ so that
 \begin{equation}\label{ordering of the kappa}
 \kappa_{1} \geq  \kappa_{2} \geq \dots \geq  \kappa_{r}.
 \end{equation}
 We can then apply \cite[Lemma~4]{41}, with $l = k_\infty$, to conclude that there exists $A(D, n) > 0$ such that 
 $$
N(T)\leq  A(D, n)^{r+1} \Bigbrac{ T^l \prod_{i = 1}^r p_i^{ \kappa_{i} - l} + \sum_{i=1}^r T^{\kappa_i}\prod_{j=i}^r p_j^{\kappa_j - \kappa_i}}.
 $$
 By \eqref{primes less than l} we have $r+1 \leq |\mathcal{P}|$.  It therefore remains to show that 
 \begin{equation}\label{summation to check}
T^l \prod_{i = 1}^r p_i^{ \kappa_{i} - l} + \sum_{i=1}^r T^{\kappa_i}\prod_{j=i}^r p_j^{\kappa_j - \kappa_i} \leq \sum_{\nu \in \mathcal{P}}  T^{k_\nu} \prod_{\substack{\mu \in \mathcal{P}\\ k_\mu < k_\nu}} \mu^{-(k_\nu - k_\mu)}.
\end{equation}
We have defined $l$ to be $k_\infty$.  Furthermore, by \eqref{primes less than l} we have $k_\mu < k_\infty$ if and only if $\mu = p_i$ for some $i$.  Therefore
$$
T^l \prod_{i = 1}^r p_i^{ \kappa_{i} - l} = T^{k_\infty} \prod_{\substack{\mu \in \mathcal{P}\\ k_\mu < k_\infty}} \mu^{-(k_\infty - k_\mu)}.
$$
Next, fix $i \in \set{1, \dots, r}$.  Then by \eqref{ordering of the kappa} we have $k_\mu < k_{p_i}$ if and only if $\mu = p_j$ for some $j > i$.  Thus
  \begin{align*}
  T^{\kappa_i}\prod_{j=i}^r p_j^{\kappa_j - \kappa_i}& =  T^{\kappa_i}\prod_{j > i} p_j^{\kappa_j - \kappa_i}\\
  & =  T^{k_{p_i}} \prod_{\substack{\mu \in \mathcal{P}\\ k_\mu < k_{p_i}}} \mu^{-(k_{p_i} - k_\mu)}.
  \end{align*}
  We have shown that each term of summation in the left-hand side of \eqref{summation to check} has an identical term of summation in the right hand side, which therefore completes the derivation of our lemma from  \cite[Lemma~4]{41}.
 \end{proof}

\section{Exponential sum estimates}\label{s:exp}

This  section is the heart  of our paper and is concerned with estimating a very general family  of  multi-dimensional exponential sums with polynomial arguments. 
We begin by introducing the following class of weight functions.

\begin{definition}[Smooth weights $\mathcal{S}^{+}(\vc)$]\label{smooth}
Let $\vc = (c, c_0, c_1, \dots)$ be an increasing infinite tuple of positive absolute constants which are \emph{super-exponential} in the sense that for any non-negative integers $i,j$ we have $c_{i+j} \geq c_ic_j$.  We define $\mathcal{S}^{+}(\vc)$ to be the set of smooth {\em weight functions} $\omega : \R^n \to [0, \infty)$ satisfying
\begin{enumerate}[(i)]
\item $\supp(\omega) \subset [-c, c]^n$;
\item for any $\vepsilon \in (\N\cup\{0\})^n$  we have $\norm{\partial^{\vepsilon} \omega}_{L^\infty(\R^n)} \leq c_{\epsilon_1 + \dots + \epsilon_n}$. 
\end{enumerate}
\end{definition}

Of central concern to us is the 
 \emph{exponential sum}
$$
S(\alpha, P):= \sum_{\vx \in \Z^n} \omega(\vx/P) e(\alpha f(\vx)),
$$
where $\omega \in  \mathcal{S}^+(\vc)$ 
and $f\in\Z[x_1, \dots, x_n]$ is the \emph{underlying polynomial}.   
Throughout, we write $f^{[k]}$ for the homogeneous part of $f$ of degree $k$.  The height of $f$, 
written $\mathrm{Height}\brac{f}$, is the 
maximum absolute value of the coefficients of $f$.
We henceforth  assume that the underlying polynomial $f$ has degree at most $d$, with leading form $f^{[d]}$ having  singular locus \eqref{singular locus} over $\Q$ of dimension $\sigma := \sigma_\infty(f^{[d]})$.  In the statement of all results in this section we assume that $\alpha \in \T$ and $q \in \N$ are primitive.  

\begin{remark}[Implicit constants]
Throughout this section, all implicit constants may depend on $\eps, d, n$ and $c_i$, where $c_i$ is a term of the super-exponential sequence appearing in $\mathcal{S}^+(\vc)$.  We determine $\vc$ for our particular choice of $\omega$ in \S \ref{major arc section}, from which it follows that  $c_i = O_{i, F}(1)$, where $F$ is the form occurring in Theorem \ref{general d improvement theorem} or \ref{small d improvement theorem}.  However, since the $c_i$ increase with $i$, we emphasise that the dependence of any implicit constant on $c_i$ is subject to the caveat that $i = O_{d, n}(1)$, which can be guaranteed in all our results.
\end{remark}

Our first estimate is the classical Weyl bound for $S(\alpha,P)$.  However, unlike the standard treatment found in, say, \cite{birch} or \cite{dav}, we don't  assume that $\mathrm{Height}\brac{f^{[d]}} = O(1)$.

\begin{lemma}[Weyl bound]\label{weyl bound}
Suppose that $\mathrm{Height}\brac{f^{[d]}} \leq H$.  Then 
\begin{align*}
\abs{\frac{S(\alpha, P)}{P^n}}^{2^{d-1}} \ll~& 
(\log P)^{n}\\
&\times \Bigbrac{ P^{1-d} + \norm{q\alpha}H + qP^{-d} + \min\bigset{Hq^{-1}, \trecip{\norm{q\alpha} P^d} }}^{\frac{n-\sigma}{d-1}}.
\end{align*}
\end{lemma}

Note that $\sigma=n$ when $\deg(f)<d$, so that this estimate is trivial in that case.
Rather than giving a suitably modified sketch of the standard proof, we opt for a detailed  account based on van der Corput differencing.  This affords us the opportunity to introduce, in a less technical context, some of the key ideas behind our later arguments.

\begin{proof}[Proof of Lemma \ref{weyl bound}]
Let $1 \leq H_1, \dots, H_{d-1} \leq P$ be parameters to be determined later.  Set $\chi(\vx) := \omega(\vx/P) e(\alpha f(\vx))$.  Changing variables and averaging gives the identity
$$
S(\alpha, P) = \recip{\floor{H_1}^n}\ \sum_{1 \leq \vh \leq H_1}\ \sum_{\vx \in \Z^n} \chi(\vx + \vh).
$$
The number of $\vx \in \Z^n$ for which there exists $1 \leq \vh \leq H_1$ with $\chi(\vx + \vh) \neq 0$ is $O\bigbrac{ P^n}$. Interchanging the order of summation and applying Cauchy--Schwarz, it follows that 
$$
 |S(\alpha, P)|^2 \ll \frac{P^n}{H_1^{2n}}\sum_{1 \leq \vh, \vh' \leq H_1} \sum_{\vx \in \Z^n} \chi(\vx+\vh)\overline{\chi(\vx+\vh')} ,
$$
Let
\begin{equation}\label{difference operators}
\omega_{\vh}(\vx) := \omega(\vx+\vh)\omega(\vx) \quad \text{ and }\quad f_{\vh}(\vx) := f(\vx +\vh) - f(\vx).
\end{equation}
Applying the triangle inequality, it follows that 
$$
|S(\alpha, P)|^2 \ll \frac{P^n}{H_1^{2n}}\sum_{\vh_1} r_{H_1}(\vh_1)\Bigabs{\sum_{\vx \in \Z^n} \omega_{\vh_1/P}(\vx/P) e(\alpha(f_{\vh_1}(\vx)))},
$$
where
$$
r_{H_1}(\vh_1) := \hash\set{(\vh,\vh') : 1 \leq \vh, \vh' \leq H_1 \text{ and } \vh - \vh' = \vh_1}. 
$$
Notice that $r_{H_1}(\vh_1) \leq H_1^n$ and $r_{H_1}(\vh_1) = 0$ if $|\vh_1| \geq H_1$.  Thus
\begin{equation}\label{base cauchy}
\abs{\frac{S(\alpha, P)}{P^n}}^2 \ll \frac{1}{H_1^n} \sum_{-H_1 < \vh_1 < H_1}\abs{\frac{S_{\vh_1}(\alpha, P)}{P^n}} ,
\end{equation}
where
\begin{equation}\label{base Sh_1}
S_{\vh_1}(\alpha, P) := \sum_{\vx \in \Z^n} \omega_{\vh_1/P}(\vx/P) e(\alpha(f_{\vh_1}(\vx))).
\end{equation}

We call the derivation of \eqref{base cauchy} the method of {\em van der Corput differencing}.
We claim that $\omega_{\vh_1/P} \in \mathcal{S}^+(\vc')$ with $c' = c$ and $c_m' = c_m 2^m$.  Since $\omega_{\vh_1/P}$ is a product of two non-negative smooth functions, it is itself non-negative and smooth.  Since one of the factors which comprise $\omega_{\vh_1/P}$ is $\omega$, we have
$$
\supp(\omega_{\vh_1/P}) \subset \supp(\omega) \subset [-c, c]^n.
$$
Finally, by the product rule for differentiation and the super-exponential nature of the $c_m$, for any $\vepsilon \in (\N\cup\{0\})^n$ with $\epsilon_1 + \dots + \epsilon_n = m$ we have
\begin{align*}
|\partial^{\vepsilon} \omega_{\vh_1/P}(\vx)| & \leq \sum_{\vlambda + \vmu = \vepsilon} \binom{\epsilon_1}{\lambda_1} \dotsm \binom{\epsilon_n}{\lambda_n} |\partial^{\vlambda}\omega(\vx+P^{-1}\vh_1) \partial^{\vmu}\omega(\vx)|\\
& \leq c_m \sum_{\vlambda + \vmu = \vepsilon} \binom{\epsilon_1}{\lambda_1} \dotsm \binom{\epsilon_n}{\lambda_n}\\
& = c_m 2^m.
\end{align*}
The claim follows.  

Let us define $\omega_{\vh_1, \dots, \vh_r}$ via \eqref{difference operators} with $\omega = \omega_{\vh_1, \dots, \vh_{r-1}}$ and $\vh= \vh_{r}$.  Then by induction and our previous claim, we have $\omega_{(\vh_1, \dots, \vh_r)/P} \in \mathcal{S}^+(\vc')$ with $\vc'$  the super-exponential sequence given by $c' = c$ and 
$
c_m' = c_m 2^{rm}.
$ 

Define $f_{\vh_1, \dots, \vh_{r}}$ analogously, so that if $g = f_{\vh_1, \dots, \vh_{r-1}}$ is defined, then we set
\begin{align*}
f_{\vh_1, \dots, \vh_r}(\vx) &:= g_{\vh_r}(\vx) = f_{\vh_1, \dots, \vh_{r-1}}(\vx+\vh_r) - f_{\vh_1, \dots, \vh_{r-1}}(\vx).
\end{align*}
Notice that
\begin{align*}
f_{\vh_1, \vh_2}(\vx) & = f(\vx+ \vh_1 + \vh_2) - f(\vx + \vh_1) - f(\vx + \vh_2) + f(\vx)\\
& = f_{\vh_2, \vh_1}(\vx).
\end{align*}
It follows from this, and the inductive definition,
that $f_{\vh_1, \dots, \vh_r}$ is invariant under permutations of the $\vh_i$.  Furthermore, by Taylor's formula we have 
\begin{equation}\label{taylor applied to f}
f^{[d-r]}_{\vh_1, \dots, \vh_r} = \vh_r \cdot \nabla f^{[d-(r-1)]}_{\vh_1, \dots, \vh_{r-1}}.
\end{equation}
Consequently, $f^{[d-r]}_{\vh_1, \dots, \vh_r}$ is linear in each $\vh_i$.

Iterating the argument that led to \eqref{base cauchy} and \eqref{base Sh_1}, we deduce that for each $1 \leq r \leq d-1$ we have
\begin{equation}\label{eq:vanilla}
\abs{\frac{S(\alpha, P)}{P^n}}^{2^r} \ll \frac{1}{(H_1\dotsm H_r)^n}\sum_{-H_1 < \vh_1< H_1} 
\hspace{-0.1cm}
\dots
\hspace{-0.1cm}
\sum_{ -H_r < \vh_r<H_r}\abs{\frac{S_{\vh_1, \dots, \vh_{r}}(\alpha, P)}{P^n}} ,
\end{equation}
where 
$$
S_{\vh_1, \dots, \vh_r}(\alpha, P) := \sum_{\vx \in \Z^n} \omega_{(\vh_1, \dots, \vh_r)/P}(\vx/P) e(\alpha(f_{\vh_1, \dots, \vh_r}(\vx))).
$$

Since $f_{\vh_1, \dots, \vh_{d-1}}^{[1]}(\vx)$ is linear in $\vx$, we have
$$
f_{\vh_1, \dots, \vh_{d-1}}(\vx) = \sum_i \brac{\partial^{\ve_i} f_{\vh_1, \dots, \vh_{d-1}}^{[1]}} x_i + c_{\vh_1, \dots, \vh_{d-1}},
$$
where $\partial^{\ve_i} f_{\vh_1, \dots, \vh_{d-1}}^{[1]}$ and $c_{\vh_1, \dots, \vh_{d-1}}$ are integers independent of $\vx$.  Define the function   $\phi : [-cP, cP]^n \cap \Z^n \to \C$ via  $\phi(\vx) = e(\alpha f_{\vh_1, \dots, \vh_{d-1}}(\vx))$. Then, in the notation of Lemma \ref{partial summation lemma}, we have 
\begin{align*}
|T_\phi(\vt)| & = \abs{\prod_{i=1}^n \sum_{-cP \leq x_i \leq t_i} e(\alpha \partial^{\ve_i} f_{\vh_1, \dots, \vh_{d-1}}^{[1]}x_i)}
\\ & \ll \prod_{i=1}^n \min\set{P, \bignorm{\alpha \partial^{\ve_i} f_{\vh_1, \dots, \vh_{d-1}}^{[1]}}^{-1}}.
\end{align*}
Let $g(\vt) = \omega_{(\vh_1, \dots, \vh_{d-1})/P}(\vt/P)$.  
For  $\epsilon_1 + \dots + \epsilon_n = m$, it follows from the chain rule that
$$
|\partial^{\vepsilon} g(\vt) | \leq  P^{-m} \norm{\partial^{\vepsilon}\omega_{(\vh_1, \dots, \vh_{d-1})/P}}_{L^\infty(\R^n)} \leq c_m2^{(d-1) m} P^{-m}.
$$
Hence partial summation (Lemma \ref{partial summation lemma}) yields the existence of  $\vepsilon \in \set{0,1}^n$ such that 
\begin{align*}
S_{\vh_1,\dots, \vh_{d-1}}(\alpha, P) & \ll   \frac{P^{\epsilon_1 + \dots + \epsilon_n}}{P^n}  \int_{[-cP, cP]^n} \bigabs{\partial^{\vepsilon}g \bigbrac{N \overline{\vepsilon} + \vt_{\vepsilon}}\ T_\phi\bigbrac{N \overline{\vepsilon} + \vt_{\vepsilon}} }\intd \vt\\
& \ll  \prod_{j=1}^n\min\set{P, \bignorm{\alpha \partial^{\ve_j} f_{\vh_1, \dots, \vh_{d-1}}^{[1]}}^{-1}} .
\end{align*}

Write $\underline{\vh}$ as a shorthand for the vector $(\vh_1, \dots, \vh_{d-1})$ and write 
$$M(\underline{\vh}):=
\bigbrac{ \partial^{\ve_j} f_{\vh_1, \dots, \vh_{d-1}}^{[1]}}_{1\leq j \leq n}.
$$
Let us view the vector $\alpha M(\underline{\vh})$ as an element of the torus $\T^n = \R^n/ \Z^n$.  Sub-dividing this torus into sub-cubes of side-length $P^{-1}$, each vector $ \alpha M(\underline{\vh})$ has $j$th coordinate lying in an interval $[\frac{r_j}{P}, \frac{r_j+1}{P})$, for some $\vr \in \Z^n$ with $0 \leq r_j < P$.  Let $R(\vr)$ deonte this region.  If $\alpha M(\underline{\vh}) \in R( \vr)$ then $\norm{\alpha M(\underline{\vh})_j} \geq r_j/P$ for each $j$.  Letting
$$
T(\vr):= \set{\underline{\vh}
 : \mbox{$\alpha M(\underline{\vh}) \in R(\vr)$  and  $|\vh_i| < H_i$ for $1\leq i \leq d-1$}},
$$ 
 we have
\begin{align*}
\sum_{|\vh_1| < H_1}\dots \sum_{|\vh_{d-1}| < H_{d-1}} &\prod_{j=1}^n\min\set{1, (P\norm{\alpha  M(\underline{\vh})_j})^{-1}}\\ & \leq \sum_{0 \leq \vr < P} \hash T( \vr) \prod_{j=1}^n \min\set{1, \tfrac{1}{r_j}} \\
& \ll  ( \log P)^n \max_{ \vr} \hash T( \vr).
\end{align*}

Define
$$
N_{\vh_1, \dots, \vh_{d-2}}( \vr) :=  \hash\set{ \vh_{d-1} \in \Z^n : \underline{\vh} \in T(\vr) }
$$
and, in a similar fashion,  let
$n_{\vh_1, \dots, \vh_{d-2}}$  denote  the number of integer vectors $\vh_{d-1}$ 
such that $|\vh_{d-1}|<H_{d-1}$ and 
 $\| \alpha M(\underline{\vh})_j\| \leq P^{-1}$
for all $j$. Now if $\vh_{d-1}$ and $\vh_{d-1}'$ are counted by 
$N_{\vh_1, \dots, \vh_{d-2}}( \vr)$ then $\vh_{d-1}'-\vh_{d-1}$ is counted by 
$n_{\vh_1, \dots, \vh_{d-2}}$, whence 
$N_{\vh_1, \dots, \vh_{d-2}}(\vr)\leq 
n_{\vh_1, \dots, \vh_{d-2}}$ for any $\vr\in \ZZ^n$.
It therefore follows that
$\hash
T(\vr) \ll \mathcal{M}_{H_1, \dots, H_{d-1}}(P^{-1}),
$
where 
$$
\mathcal{M}_{H_1, \dots, H_{d-1}}(P^{-1}):=
 \hash \set{\underline{\vh} \in \Z^{(d-1)n} : |\vh_i| < H_i  \text{ and } \norm{\alpha M(\underline{\vh})_j} < P^{-1} }.
$$
Combining these deliberations, we deduce that 
\begin{equation}\label{pre-shrinking weyl estimate}
\abs{\frac{S(\alpha, P)}{P^n}}^{2^{d-1}} \ll \brac{\frac{\log P}{H_1\dotsm H_{d-1}}}^n \mathcal{M}_{H_1, \dots, H_{d-1}}(P^{-1}).
\end{equation}

Next, we claim that the linear map $\vh_i \mapsto  M(\underline{\vh})$ has a symmetric matrix (with respect to the standard basis).  Since $M(\underline\vh)=M(\vh_1, \dots, \vh_{d-1})$ is invariant under permutations of the $\vh_i$, it suffices to establish the claim when $i = d-1$.  By \eqref{taylor applied to f} and linearity of differentiation, we have
\begin{equation}\label{e_i e_j}
\begin{split}
M(\vh_1, \dots, \vh_{d-1}) & = \Bigbrac{ \partial^{\ve_i}\bigbrac{ \vh_{d-1} \cdot \nabla  f^{[2]}_{\vh_1, \dots, \vh_{d-2}}}}_{1\leq i\leq n} \\
& =  \vh_{d-1}\cdot\Bigbrac{ \partial^{\ve_i + \ve_j}  f^{[2]}_{\vh_1, \dots, \vh_{d-2}}}_{\substack{1\leq i\leq n \\ 1 \leq j \leq n}} .
\end{split}
\end{equation}
Since $\ve_i + \ve_j=  \ve_j + \ve_i$, the claim follows

We can therefore apply the shrinking lemma (Lemma \ref{shrinking lemma}) to each block of variables $\vh_i$ in $\underline{\vh}$ and conclude that for any $\theta_1, \dots, \theta_{d-1} \in (0, 1]$ we have
\begin{align*}
\mathcal{M}_{H_1, \dots, H_{d-1}}(P^{-1}) \ll \frac{\mathcal{M}_{\theta_1H_1, \dots, \theta_{d-1}H_{d-1}}(\theta_1 \dotsm\theta_{d-1} P^{-1})}{(H_1^{-1} + \theta_1)^{n}\dots (H_{d-1}^{-1} + \theta_{d-1})^{n}}.
\end{align*}
By iterating \eqref{taylor applied to f}, one can check that we have the formula
\begin{align*}
 M(\vh_1, \dots, &\vh_{d-1})_j =\\ &\sum_{1 \leq i_1 \leq n} \dots \sum_{1 \leq i_{d-1} \leq n} h_{1, i_1} \dotsm h_{d-1, i_{d-1}} \partial^{\ve_{i_1} + \dots + \ve_{i_{d-1}}+ \ve_j} f^{[d]}.
\end{align*}
Hence there exists a constant $C = O_{n, d}(1)$ such that if $|\vh_i| < \theta_i H_i$ for all $i$ then
$$
| M(\vh_1, \dots, \vh_{d-1})_j| < CH (\theta_1 H_1)\dotsm (\theta_{d-1} H_{d-1}).
$$
Let us choose $\theta_1, \dots, \theta_{d-1} \in (0, 1]$ so that 
\begin{equation}\label{theta product def}
\theta_1 \dotsm \theta_{d-1}=\min\set{1, \tfrac{1}{2\norm{q\alpha}CHH_1\dotsm H_{d-1}},  \tfrac{P}{2q}, \max\set{\tfrac{q}{CHH_1\dotsm H_{d-1}}, \norm{q\alpha} P}}.
\end{equation}
It follows that if $|\vh_i| < \theta_iH_i$ for all $i$  and $\norm{\alpha M(\underline{\vh})_j} < \theta_1 \dotsm \theta_{d-1} P^{-1}$ for all $j$, then
\begin{enumerate}[(i)]
\item  $|M(\underline{\vh})_j| < \trecip{2} \norm{q\alpha}^{-1}$;
\item  $\norm{\alpha M(\underline{\vh})_j} < \recip{2} q^{-1}$;
\item  $|M(\underline{\vh})_j| < q$  or   $\norm{\alpha M(\underline{\vh})_j} < \norm{q\alpha}$.
\end{enumerate}
Applying Lemma \ref{heathbrown lemma}, we deduce that $M(\underline{\vh})_j = 0$
for $j= 1, \dots, n$.
 Incorporating this into \eqref{pre-shrinking weyl estimate} we obtain the estimate
\begin{equation}\label{post-shrinking weyl estimate}
\abs{\frac{S(\alpha, P)}{P^n}}^{2^{d-1}} \ll (\log P)^n \frac{  L_f(\theta_1H_1, \dots, \theta_{d-1} H_{d-1})}{(1 + \theta_1H_1)^{n}\dots (1 + \theta_{d-1}H_{d-1})^{n}},
\end{equation}
with $\theta_1 \dotsm \theta_{d-1}$ as in \eqref{theta product def} and where
$$
L_f(H_1, \dots, H_{d-1}):= \hash \set{\underline{\vh} \in \Z^{(d-1)n} : |\vh_i| < H_i  \text{ and } \nabla f^{[1]}_{\vh_1, \dots, \vh_{d-1}} =\mathbf{0} }.
$$
We are therefore led to the estimation of $L_f(H_1, \dots, H_{d-1})$.

Using the notation \eqref{sigma dimension notation}, we may partition the set of $\vh_1 \in (-H_1, H_1)^n \cap \Z^n$ according to the value of $\sigma_{\infty}(f_{\vh_1}^{[d-1]})$.  Using this and the pigeon-hole principle, we deduce the existence of an integer $s_1\in [0,n]$, a set $\mathcal{H}_1 \subset (-H_1, H_1)^n \cap \Z^n$ and $\vh_1 \in \mathcal{H}_1$ such that both of the following hold:
\begin{enumerate}[(i)]
\item  For each $\vh_1' \in \mathcal{H}_1$ we have $\sigma_{\infty}(f_{\vh_1'}^{[d-1]}) = s_1$.
\item  $L_f(H_1, \dots, H_{d-1}) \ll |\mathcal{H}_1| L_{f_{\vh_1}}(H_2, \dots, H_{d-1}).$
\end{enumerate}

Iterating this process, we can find integers $s_1, \dots, s_{d-2}$, sets $\mathcal{H}_1, \dots, \mathcal{H}_{d-2}$ and fixed elements $\vh_i \in \mathcal{H}_i$ such that
\begin{enumerate}[(i)]
\item  For each $\vh_i' \in \mathcal{H}_i$ we have $\sigma_{\infty}(f_{\vh_1, \dots, \vh_{i-1}, \vh_i'}^{[d-i]}) = s_i$.
\item  
\begin{equation}\label{first pigeon-hole bound 0}
L_f(H_1, \dots, H_{d-1}) \ll |\mathcal{H}_1| \dotsm |\mathcal{H}_{d-2}| L_{f_{\vh_1, \dots, \vh_{d-2}}}(H_{d-1}).
\end{equation}
\end{enumerate}
By Euler's identity, the singular locus of $f^{[2]}_{\vh_1, \dots, \vh_{d-2}}$ is the set of $\vx$ such that for each $i \in \set{1, \dots, n}$ we have  
\begin{align*}
0 & = \partial^{\ve_i} f^{[2]}_{\vh_1, \dots, \vh_{d-2}}(\vx)
	 = \sum_{j=1}^n \brac{\partial^{\ve_i + \ve_j} f_{\vh_1, \dots, \vh_{d-2}}^{[2]}} x_j.
\end{align*}
Hence by \eqref{e_i e_j}, the nullity of the linear map $\vh_{d-1} \mapsto M(\vh_1, \dots, \vh_{d-1})$ (over $\Q$) coincides with the dimension $s_{d-2} = \sigma_\infty(f^{[2]}_{\vh_1, \dots, \vh_{d-2}})$.  By elementary linear algebra (inducting on the nullity) one can show that for a given linear map $T:\C^n \to \C^m$ and any $\vy \in \C^m$ we have the uniform estimate
$$
\hash\set{\vx \in (-P, P)^n\cap \Z^n  : T\vx = \vy} \ll_n P^{n- \rank(T)}.
$$
Hence for $H_1 \geq 1$ we have
\begin{equation}\label{linear L estimate}
L_{f_{\vh_1, \dots, \vh_{d-2}}}(H_{d-1}) \ll H_{d-1}^{s_{d-2}}.
\end{equation}

For each $1\leq i \leq d-2$, define the algebraic variety
$$
X_{i, \infty} := \set{ \vy \in \A_{\Q}^n : \sigma_\infty(f^{[d-i]}_{\vh_1, \dots, \vh_{i-1}, \vy})  \geq s_{i}}.
$$
From \eqref{taylor applied to f} we have $f^{[d-i]}_{\vh_1, \dots, \vh_{i-1}, \vy} = \vy \cdot \nabla f^{[d-i+1]}_{\vh_1, \dots, \vh_{i-1}}$.  We may therefore employ Lemma \ref{dimension bound lemma} with $G = f^{[d-i+1]}_{\vh_1, \dots, \vh_{i-1}}$ to deduce that $X_{i, \infty}$ is an affine algebraic variety defined by $O_{d,n}(1)$ equations of degree $O_{d,n}(1)$.  Moreover, setting $s_0 := \sigma$, we see that the dimension of $X_{i, \infty}$ is at most $n -  (s_{i} - s_{i-1})^+$.  Taking $\mathcal{P} = \set{\infty}$ and $k_\infty = n - (s_{i} - s_{i-1})^+$ in Lemma \ref{counting lemma}, we conclude that for $H_i \geq 1$ we have 
$$
|\mathcal{H}_i|  \ll   H_i^{n- (s_{i} - s_{i-1})^+} .
$$

Combining this estimate for $|\mathcal{H}_i|$ with 
\eqref{first pigeon-hole bound 0} and  \eqref{linear L estimate}, 
we deduce that for $H_i \geq 1$ there exist integers $s_1, \dots, s_{d-2}$ such that
$$
L_f(H_1, \dots, H_{d-1}) \ll H_1^{n-(s_1 - s_0)^+} \dotsm H_{d-2}^{n- (s_{d-2}- s_{d-3})^+} H_{d-1}^{ s_{d-2}}.
$$
Set $B_i := 1 + \theta_i H_i$.  Using this in \eqref{post-shrinking weyl estimate}, we see that for any $ H_1, \dots, H_{d-1}$ belonging to the interval $[1,P]$ and any $\theta_1, \dots, \theta_{d-1} \in (0,1]$ whose product is equal to \eqref{theta product def}, we have
$$
 \abs{\frac{S(\alpha, P)}{P^n}}^{2^{d-1}}  \ll (\log P)^nB_1^{-(s_1 - s_0)^+} \dotsm B_{d-2}^{- (s_{d-2}- s_{d-3})^+} B_{d-1}^{-(n- s_{d-2})}.
$$

As we have no control over the values of the integers $s_i$, to proceed any further we must impose the condition that
$$
\theta_1 H_1 = \theta_2 H_2 = \dots = \theta_{d-1} H_{d-1}.
$$
Then since $s_0 = \sigma$ we have
\begin{align*}
B_1^{-(s_1 - s_0)^+} \dotsm B_{d-2}^{- (s_{d-2}- s_{d-3})^+} B_{d-1}^{-(n- s_{d-2})} & \leq B_1^{-(n- \sigma)}.
\end{align*}
Notice that $B_1\geq \theta_1H_1$. Moreover, writing $\widetilde H=H_1\dots H_{d-1}$ it follows from   \eqref{theta product def} that
\begin{align*}
\theta_1 H_1 =~& \brac{\theta_1 \dotsm \theta_{d-1} \widetilde H}^{\recip{d-1}} \\
	\asymp~& \min\set{\widetilde H, \tfrac{1}{\norm{q\alpha}H},  \tfrac{P\widetilde H}{q}, \max\set{\tfrac{q}{H}, 
	\widetilde H \norm{q\alpha} P}}^{\recip{d-1}}.
\end{align*}
Thus we arrive at the estimate
\begin{align*}
& \abs{\frac{S(\alpha, P)}{P^n}}^{2^{d-1}}    \ll 
(\log P)^{n}\max\Bigset{\trecip{\widetilde H}, \norm{q\alpha}H,  \tfrac{q}{P\widetilde H}, \min\bigset{\tfrac{H}{q}, \trecip{\widetilde H \norm{q\alpha} P}}}^{\frac{n-\sigma}{d-1}}.
\end{align*}
This bound is minimised by taking $H_1 = \dots = H_{d-1} = P$, which yields
$$
 \abs{\frac{S(\alpha, P)}{P^n}}^{2^{d-1}}  
  \ll (\log P)^{n}\max\Bigset{\trecip{P^{d-1}}, \norm{q\alpha} H, \tfrac{q}{P^d}, \min\bigset{\tfrac{H}{q}, \trecip{\norm{q\alpha}P ^d}}}^{\frac{n-\sigma}{d-1}},
$$
which thereby  completes the proof of the lemma.
\end{proof}

When the exponential sum $S(\alpha,P)$ involves a cubic polynomial one can get better estimates by applying Poisson summation instead of repeated applications of van der Corput differencing. This is one of the key innovations in Heath-Brown \cite{hb10}, for example, and it also played a critical role in \cite{41}.  

Given a fixed positive integer $q$ we reserve the symbols $b, c_1$ and $c_2$ for the following quantities
\begin{equation}\label{bcd defn}
b := \prod_{\substack{ p^e \| q\\ e \leq 2}} p^e, \qquad c_1 := \prod_{\substack{ p^e \| q\\ e > 2}} p^{\floor{\frac{e}{2}}}, \qquad c_2 := \prod_{\substack{ p^e \| q\\ e > 2,\ 2 \nmid e}} p,
\end{equation}
so that $q = bc_1^2c_2$.  Define the \emph{$r$-values} of $f$ with respect to $q$ to be the numbers 
\begin{equation}\label{r values notation}
r_i = r^{[d]}_i(f, q) := \prod_{\substack{p^e || bd \\ \sigma_p(f^{[d]}) = i}} p^e, \qquad R_\zeta = R^{[d]}_\zeta(f, q) := \prod_{\zeta \leq i \leq n} r_i^{i-\zeta}.
\end{equation}
Define the relative height of $f$ with respect to $P$ at scale $d$ by
$$
\mathrm{Height}_{P,d}\brac{f} := \mathrm{Height}\brac{P^{-d} f(P\vx)}.
$$
Notice that for $P \geq 1$ and $d \geq \deg(f)$ we have
$$
\mathrm{Height}\brac{f^{[d]}} = \mathrm{Height}_{P,d}\brac{f^{[d]}} \leq \mathrm{Height}_{P,d}\brac{f} \leq \mathrm{Height}\brac{f}.
$$
Bearing this notation in mind, the  following result is a reformulation of the key estimate in Browning and Heath-Brown \cite{41}.

\begin{lemma}[Cubic Poisson bound]\label{cubic poisson bound}
Suppose that $\deg(f) \leq 3$ and  $H$ is such that 
$\mathrm{Height}_{P,3}\brac{f} \leq H \leq P^{O(1)}$.
Then for any $q \leq P^2$, $\norm{q\alpha} \leq P^{-1}$ and $\zeta$ in the range $\sigma_\infty(f^{[3]}) \leq \zeta \leq n$ we have 
\begin{align*}
S(\alpha, P) 
&\ll
R^{[3]}_\zeta(f, q)^{\frac{1}{2}} P^{n+\eps}\Bigbrac{ \tfrac{\sqrt{q}}{P} + \sqrt{\norm{q\alpha} PH} + \tfrac{H^{1/6}
\min\set{c_1, c_2H}^{\frac{1}{6}}
}{b^{\frac{1}{2}}(c_1c_2)^{\frac{1}{3}}} }^{n-\zeta}.
\end{align*}
\end{lemma}

\begin{proof}
The statement of the lemma is trivial if $\deg(f)<3$ since then we are obliged to take 
$\zeta
=\sigma_\infty(f^{[3]})=n$.  Suppose henceforth that $\deg(f)=3$.
Let
$$
V := \max\set{qP^{-1}, \sqrt{q\norm{q\alpha} HP}}
$$
and
$$
W:= V + \min\set{(c_1^2c_2H)^{\recip{3}}, (c_1V)^{\recip{2}} + c_1^{\frac{5}{6}}H^{\frac{1}{6}}}.
$$
Then by \cite[Prop.~2]{41} we have
$$
S(\alpha, P) \ll R_\zeta^{\frac{1}{2}}P^{n+\eps} \brac{Wq^{-\recip{2}}}^{n-\zeta}.
$$
Using $q = bc_1^2c_2$, we see that
$$
Wq^{-\recip{2}} \ll Vq^{-\recip{2}} + \min\set{ \frac{H^{\recip{3}}}{b^{\recip{2}}c_1^{\recip{3}}c_2^{\recip{6}}}, \sqrt{\frac{V}{bc_1c_2}}} + \min\set{ \frac{H^{\recip{3}}}{b^{\recip{2}}c_1^{\recip{3}}c_2^{\recip{6}}}, \frac{H^{\recip{6}}}{b^{\recip{2}}c_1^{\recip{6}}c_2^{\recip{2}}}},
$$
with 
$$
Vq^{-\recip{2}} \leq \frac{\sqrt{q}}{P} + \sqrt{\norm{q\alpha}HP}.
$$
It therefore suffices to establish that 
\begin{equation}\label{V condition to prove i}
\min\set{ \frac{H^{\recip{3}}}{b^{\recip{2}}c_1^{\recip{3}}c_2^{\recip{6}}}, \sqrt{\frac{V}{bc_1c_2}}}  \ll Vq^{-\recip{2}} + \min\set{ \frac{H^{\recip{3}}}{b^{\recip{2}}c_1^{\recip{3}}c_2^{\recip{6}}}, \frac{H^{\recip{6}}}{b^{\recip{2}}c_1^{\recip{6}}c_2^{\recip{3}}}}.
\end{equation}
If 
$$
 \frac{H^{\recip{3}}}{b^{\recip{2}}c_1^{\recip{3}}c_2^{\recip{6}}} \leq \frac{H^{\recip{6}}}{b^{\recip{2}}c_1^{\recip{6}}c_2^{\recip{3}}}
$$
then \eqref{V condition to prove i} follows immediately.  We may therefore assume that the opposite inequality holds, or equivalently (after re-arrangement), that
\begin{equation}\label{Hdc}
Hc_2 > c_1.
\end{equation}
In this case, \eqref{V condition to prove i} follows if we can prove that
$$
\sqrt{\frac{V}{bc_1c_2}}  \leq \max\set{\frac{V}{b^{\recip{2}} c_1 c_2^{\recip{2}}} , \frac{H^{\recip{6}}}{b^{\recip{2}}c_1^{\recip{6}}c_2^{\recip{3}}}}.
$$
By the trivial inequality $\max\set{X, Y} \geq \sqrt{XY}$, the right-hand side is at least 
$$
\brac{\frac{V}{bc_1c_2}}^{\recip{2}} \brac{\frac{Hc_2}{c_1}}^{\recip{12}},
$$
so that the desired condition now follows from \eqref{Hdc}.
\end{proof}

In Lemma \ref{weyl bound} we gave a detailed account of how $d-1$ applications of van der Corput differencing can be used to transform the exponential sum $S(\alpha,P)$ into a family of linear exponential sums indexed by  
$H_1,\dots,H_{d-1}\in [1,P]$, which we could ultimately estimate rather well. In the final stages of the argument we made the specialisation $H_1=\dots=H_{d-1}=P$, which brought us to the usual Weyl estimate  (as found in \cite{dav}). In the next result, we consider the effect of van der Corput differencing $d-k$ times only, for any $1\leq k\leq d$.
Rather than ending up with something of the form \eqref{eq:vanilla}, with $r=d-k$, it turns out that it will be more efficient to make a judicious application of the pigeon-hole principle at each differencing step separately, in order to control the singular locus of the underlying polynomial.

\begin{lemma}[van der Corput process]\label{van der corput process}
Suppose that $q \leq P^{O(1)}$ and that $\mathrm{Height}\brac{f} = O(1)$.  For each $1 \leq Q \leq P$ and $1 \leq k \leq d$, there exists an exponential sum $T(\alpha,P) $, with underlying polynomial $g$ of degree at most $k$ and $\mathrm{Height}_{P,k} (g)\ll Q^{2 - 2^{k+1 - d}}$, such that for some $\zeta \geq \max\set{\sigma, \sigma_{\infty}(g^{[k]})}$ we have
$$
\abs{\frac{S(\alpha,P)}{P^n}}^{2^{d-k}} \ll \frac{P^{\eps}}{Q^{\zeta-\sigma} \sqrt{R^{[k]}_\zeta(g, q)}} \abs{\frac{T(\alpha,P)}{P^n}}.
$$
Moreover, $T(\alpha,P)$ has weight in $\mathcal{S}^+(\vc')$ where $\vc'$ is the super-exponential sequence given by $c' = c$ and 
$
c_m' = c_m2^{(d-k)m}.
$
\end{lemma}

\begin{proof}
After a single iteration of van der Corput differencing, as in \eqref{base cauchy} and \eqref{base Sh_1}, we deduce that for any real $H_1 \in [1, P]$ we have the bound
$$
\abs{\frac{S(\alpha, P)}{P^n}}^2 \ll \frac{1}{H_1^n} \sum_{-H_1 < \vh_1 < H_1}\abs{\frac{S_{\vh_1}(\alpha, P)}{P^n}} ,
$$
where
$$
S_{\vh_1}(\alpha, P) := \sum_{\vx \in \Z^n} \omega_{\vh_1/P}(\vx/P) e(\alpha(f_{\vh_1}(\vx))).
$$

With $b$ and $c_2$ defined as in \eqref{bcd defn}, let  
$$
 \mathcal{V} := \set{p : p \mid bc_2} \cup \set{\infty}.
 $$
Notice that 
 $$
 \bigbrac{\sigma_{\nu}(f_{\vh_1}^{[d-1]})}_{\nu \in \mathcal{V}} \in \prod_{\nu \in \mathcal{V}} \set{0, 1, \dots, n}.
 $$
Since $|\mathcal{V}| \leq \omega(q) + 1$, we see that there are at most $O_{n, \eps}(q^\eps)$ choices for $ \bigbrac{\sigma_{\nu}(f_{\vh_1}^{[d-1]})}_{\nu \in \mathcal{V}}$.  It follows from the pigeon-hole principle that there exists a tuple of integers $\vs_1$ and a set $\mathcal{H}_1 \subset (-H_1, H_1)^n \cap \Z^n$ such that both of the following hold

\begin{enumerate}[(i)]
\item  For each $\vh_1 \in \mathcal{H}_1$ we have $ \bigbrac{\sigma_{\nu}(f_{\vh_1}^{[d-1]})}_{\nu \in \mathcal{V}} = \vs_1$.
\item  
$$
\sum_{-H_1 < \vh_1 < H_1 }|S_{\vh_1}(\alpha, P)| \ll q^\eps \sum_{\vh_1 \in \mathcal{H}_1} |S_{\vh_1}(\alpha, P)|.
$$
\end{enumerate}
Since $q \leq P^{O(1)}$, we deduce that there exists $\vh_1 \in \mathcal{H}_1$ satisfying
$$
\abs{\frac{S(\alpha, P)}{P^n}}^2 \ll P^\eps\frac{ |\mathcal{H}_1|}{H_1^n} \abs{\frac{S_{\vh_1}(\alpha, P)}{P^n}} .
$$

Next, 
let $r := d-k$. Applying the van der Corput differencing process to $S_{\vh_1}(\alpha, P)$ and iterating, we deduce that for any $1\leq H_1, \dots, H_r \leq P$ there exist sets $\mathcal{H}_i \subset (-H_i, H_i)^n \cap \Z^n$ and elements $\vh_i \in \mathcal{H}_i$  such that
\begin{equation}\label{iterated cauchy}
\begin{split}
\abs{\frac{S(\alpha, P)}{P^n}}^{2^r} \ll~& P^\eps\brac{\frac{ |\mathcal{H}_1|}{H_1^n}}^{2^{r-1}} \brac{\frac{ |\mathcal{H}_2|}{H_2^n}}^{2^{r-2}}\dotsm \brac{\frac{ |\mathcal{H}_r|}{H_r^n}}\\
&\times
\abs{\frac{S_{\vh_1, \dots, \vh_{r}}(\alpha, P)}{P^n} }.
\end{split}
\end{equation}
Moreover, there exist tuples of integers $\vs_i=(s_{i,\nu})_{\nu\in \mathcal{V}}$ such that for any $\vh_i' \in \mathcal{H}_i$ we have 
\begin{equation}\label{pigeon holed singluar locus dimension}
\bigbrac{\sigma_\nu(f^{[d-i]}_{\vh_1, \dots, \vh_{i-1}, \vh_i'})}_{\nu \in \mathcal{V}} = \vs_i.
\end{equation}

For notational convenience, let us define
$$
\vs_0 := \bigbrac{\sigma_\nu(f^{[d]})}_{\nu \in \mathcal{V}}.
$$
For any prime $p$ we have $s_{0,p} \geq s_{0, \infty}$, with strict inequality for only finitely many primes.  It follows that there exists a constant $C = O_{f}(1)$ such that $C \geq d$ and for any $p > C$ we have $s_{0,p} = s_{0, \infty}$.  Set 
$$
\mathcal{V}_C := \set{ \nu \in \mathcal{V} : \nu > C}.
$$ 
For each $\nu \in \mathcal{V}_C$ define the sets
$$
X_{\nu} := \set{ \vy \in \A_{\F_\nu}^n : \sigma_\nu(f^{[d-r]}_{\vh_1, \dots, \vh_{r-1}, \vy})  \geq s_{r,\nu}}.
$$
Notice that $f^{[d-r]}_{\vh_1, \dots, \vh_{r-1}, \vy} = \vy \cdot \nabla f^{[d-(r-1)]}_{\vh_1, \dots, \vh_{r-1}}$ by \eqref{taylor applied to f}.  We may therefore employ Lemma \ref{dimension bound lemma} with $G = f^{[d-(r-1)]}_{\vh_1, \dots, \vh_{r-1}}$, and deduce that $X_{\nu}$ is an affine algebraic variety defined by $O_{d,n}(1)$ equations of degree $O_{d,n}(1)$.  Moreover the dimension of $X_{\nu}$ is at most $n -  (s_{r,\nu} - s_{r-1, \nu})^+$.  Taking $\mathcal{P} = \mathcal{V}_C$ and $k_\nu = n - (s_{r,\nu} - s_{r-1, \nu})^+$ in Lemma \ref{counting lemma}, we conclude that 
$$
\frac{|\mathcal{H}_r|}{H_r^n}  \ll q^\eps  \sum_{\nu \in \mathcal{V}_C}  H_r^{- (s_{r,\nu} - s_{r-1, \nu})^+} \prod_{\substack{p \in \mathcal{V}_C \\ s_{r,p} - s_{r-1, p} \geq\\ (s_{r,\nu} - s_{r-1, \nu})^+ }} p^{-(s_{r,p} - s_{r-1, p}) + (s_{r,\nu} - s_{r-1, \nu})^+ }.
$$
By the pigeon-hole principle, we see that there exists  $\nu \in \mathcal{V}_C$ such that on setting $s_r := s_{r, \nu}$ and $t_{r} := s_{r-1, \nu}$ we have
\begin{equation}\label{pigeon-holed H_r}
\frac{|\mathcal{H}_r|}{H_r^n}  \ll q^{2\eps}\    H_r^{- (s_r - t_{r})^+} \prod_{\substack{p \in \mathcal{V}_C \\ s_{r,p} - s_{r-1, p} \geq\\ (s_r - t_{r})^+ }} p^{-(s_{r,p} - s_{r-1, p}) + (s_r - t_{r})^+ }.
\end{equation} 
Next, define the set
$$
\mathcal{V}^{(r-1)}_{C} := \set{ \mu \in \mathcal{V}_C : s_{r-1,\mu} \geq t_r}.
$$
Repeating the argument leading to  \eqref{pigeon-holed H_r}, we deduce that there is  a $\mu \in \mathcal{V}_{C}^{(r-1)}$ such that on setting $s_{r-1} : = s_{r-1, \mu}$ and $t_{r-1}: = s_{r-2, \mu}$, we have
$$
\frac{|\mathcal{H}_{r-1}|}{H_{r-1}^n}  \ll q^{\eps}\    H_{r-1}^{- (s_{r-1} - t_{r-1})^+} \prod_{\substack{p \in \mathcal{V}_{C}^{(r-1)} \\ s_{r-1,p} - s_{r-2, p} \geq\\ (s_{r-1} - t_{r-1})^+ }} p^{-(s_{r-1,p} - s_{r-2, p}) + (s_{r-1} - t_{r-1})^+ }.
$$

Let us write $\mathcal{V}^{(r)}_C$ for $\mathcal{V}_C$ and $t_{r+1} := 0$.  Iterating the above process, we obtain integers $s_i$ and $t_i$ for $1\leq i \leq r$, with $s_{i} \geq t_{i+1}$, such that on setting
$$
\mathcal{V}^{(i)}_{C} = \set{ \nu \in \mathcal{V}_C : s_{i,\nu} \geq t_{i+1}},
$$
we have the bound
\begin{equation}\label{pigeon-holed H_i}
\frac{|\mathcal{H}_{i}|}{H_{i}^n}  \ll q^{\eps}\    H_{i}^{- (s_{i} - t_{i})^+} \prod_{\substack{p \in \mathcal{V}_{C}^{(i)} \\ s_{i,p} - s_{i-1, p} \geq\\ (s_{i} - t_{i})^+ }} p^{-(s_{i,p} - s_{i-1, p}) + (s_{i} - t_{i})^+ }.
\end{equation}

Let us set $\zeta_0 := \sigma$ and for $i\geq 1$ set
\begin{align*}
\zeta_i  :&=  (s_i - t_{i})^+ + \zeta_{i-1}\\
	& = (s_i - t_{i})^+ + (s_{i-1} - t_{i-1})^+ +\dots + (s_1 - t_{1})^+ + \sigma.
\end{align*}
Notice that $s_{i, \nu} \geq s_{i, \infty}$ for all $i$ and $\nu$.  Also $t_1 = s_{0, \nu}$ for some $\nu \in \mathcal{V}_C$, and by our choice of $C$ this means that $t_1 = \sigma$.  
We claim that 
\begin{equation}\label{eta_i estimate}
\zeta_j \geq s_j  \geq \max\set{t_{j+1}, s_{j, \infty}}, \quad \mbox{for $0\leq j\leq r$}.
\end{equation}
The second inequality follows since  $s_j \geq t_{j+1}$ for each $j \leq r$. To see the first inequality we argue by induction on $j$, the case $j=0$ being trivial. 
For $j>0$ we need to show that 
$\zeta_j=
(s_j - t_{j})^++\zeta_{j-1}\geq s_j. $  
Now the  induction hypothesis implies that $\zeta_{j-1}\geq s_{j-1}$. 
If $s_j\geq t_j$ then 
$\zeta_j\geq s_j-t_j+s_{j-1}\geq s_j$.
If, on the other hand, 
$s_j< t_j$ then 
$\zeta_j\geq s_{j-1}\geq t_j>s_j$.
This therefore establishes \eqref{eta_i estimate}.

The estimate \eqref{pigeon-holed H_i} now becomes
$$
\frac{|\mathcal{H}_i|}{H_i^n} \ll q^{\eps}\  H_i^{- (\zeta_i - \zeta_{i-1})} \prod_{\substack{p \in \mathcal{V}^{(i)}_C \\ s_{i,p} - s_{i-1, p} \geq\\ \zeta_i - \zeta_{i-1} }} p^{-(s_{i,p}-\zeta_i) + (s_{i-1, p} -  \zeta_{i-1}) }.
$$
An expression of the form $\prod_{p \in \mathcal{P}} p^{e_p}$ is minimised by taking $\mathcal{P}=\set{p : e_p < 0}$ and maximised by taking $\mathcal{P} = \set{p : e_p > 0}$.  Therefore
\begin{align*}
\prod_{\substack{p \in \mathcal{V}^{(i)}_C \\ s_{i,p} - s_{i-1, p} \geq\\ \zeta_i - \zeta_{i-1} }}  p^{-(s_{i,p}-\zeta_i) + (s_{i-1, p} -  \zeta_{i-1}) }  
&\leq \prod_{\substack{p \in \mathcal{V}^{(i)}_C \\ s_{i,p}  \geq \zeta_i  }} p^{-(s_{i,p}-\zeta_i) + (s_{i-1, p} -  \zeta_{i-1}) }\\
&\leq 
\prod_{\substack{p \in \mathcal{V}_C \\ s_{i,p}  \geq\max\set{ \zeta_i, t_{i+1}}  }} 
\hspace{-0.4cm}
p^{-(s_{i,p} - \zeta_i)} 
\hspace{-0.3cm}
\prod_{\substack{p \in \mathcal{V}_C \\ s_{i-1,p}  \geq \zeta_{i-1}  }} p^{(s_{i-1, p}  - \zeta_{i-1}) }.
\end{align*}
By \eqref{eta_i estimate} we have $\max\set{\zeta_i, t_{i+1}} = \zeta_i$.  Re-setting $\eps$ and using the estimate $q \leq P^{O(1)}$, we see that 
$$
\frac{|\mathcal{H}_i|}{H_i^n} \ll P^{\eps}\  H_i^{ - (\zeta_i -\zeta_{i-1})} \prod_{\substack{p \in \mathcal{V}_C \\ s_{i,p}  > \zeta_i}} p^{-(s_{i,p} - \zeta_i) }\prod_{\substack{p \in \mathcal{V}_C \\ s_{i-1,p}  > \zeta_{i-1}  }} p^{(s_{i-1, p}  - \zeta_{i-1}) }.
$$

The above process produces a sequence $\zeta_r \geq \dots \geq \zeta_1 \geq \zeta_0 = \sigma$ such that, on setting 
$$
L_i := \prod_{\substack{p \in \mathcal{V}_C \\ s_{i,p}  > \zeta_i }} p^{s_{i,p} - \zeta_i },
$$
we have 
$$
\frac{|\mathcal{H}_i|}{H_i^n} \ll P^{\eps}\  H_i^{ - (\zeta_i -\zeta_{i-1})}\frac{ L_{i-1} }{L_{i}}.
$$
Hence 
\begin{align*}
\brac{\frac{ |\mathcal{H}_1|}{H_1^n}}^{2^{r-1}} \brac{\frac{ |\mathcal{H}_2|}{H_2^n}}^{2^{r-2}}& \hspace{-0.1cm}
\dotsm \brac{\frac{ |\mathcal{H}_r|}{H_r^n}}\\ & \ll   \frac{P^\eps L_0^{2^{r-1}}L_r^{-1}}{L_1^{2^{r-2}} L_2^{2^{r-3}} \dotsm L_{r-1} }\prod_{i=1}^{r} H_i^{(\zeta_{i-1}-\zeta_i ) 2^{r-i}} .
\end{align*}
For each $\nu \in \mathcal{V}_C$ we have $s_{0, \nu} = s_{0, \infty} = \sigma = \zeta_0$, so  that $L_0 = 1$.  Also for all $1\leq i\leq r-1$ we have $L_i \geq 1$.  Therefore the left hand side is
$$
\ll P^\eps H_1^{(\zeta_{0}-\zeta_1) 2^{r-1}}H_2^{(\zeta_{1}-\zeta_2) 2^{r-2}}\dotsm H_r^{(\zeta_{r-1}-\zeta_r) } \prod_{\substack{p \in \mathcal{V}_C \\ s_{r,p}  > \zeta_r }} p^{-(s_{r,p} - \zeta_r)} .
$$
Let us take 
\begin{equation}\label{eq:choose H}
H_j=Q^{\frac{1}{2^{r-j}}}, \quad \mbox{for $1\leq j\leq r$}.
\end{equation}
Then we deduce that
\begin{equation}\label{H_i product bound II}
\brac{\frac{ |\mathcal{H}_1|}{H_1^n}}^{2^{r-1}} \brac{\frac{ |\mathcal{H}_2|}{H_2^n}}^{2^{r-2}}\dotsm \brac{\frac{ |\mathcal{H}_r|}{H_r^n}} \ll P^\eps\ Q^{\zeta_0-\zeta_r} \prod_{\substack{p \in \mathcal{V}_C \\ s_{r,p}  > \zeta_r }} p^{-(s_{r,p} - \zeta_r)} .
\end{equation}

Let $g = f_{\vh_1, \dots, \vh_r}$ with each $\vh_i$ the fixed element of $\mathcal{H}_i$ determined by the van der Corput process \eqref{iterated cauchy}. Put $\zeta := \zeta_r$ and $s := s_{r, \infty}$, so $\zeta \geq s$ by  \eqref{eta_i estimate}.  Recall that $r = d-k$ so that by \eqref{pigeon holed singluar locus dimension} we have $\sigma_\nu(g^{[k]}) = s_{r, \nu}$.
In  the notation \eqref{r values notation} it therefore follows that
$$
r_\zeta := r^{[k]}_\zeta(g, q) = \prod_{\substack{ p^e \| bd\\ s_{r,p} = \zeta}}  p^e.
$$
Set 
$$
\rho_\zeta := \prod_{\substack{ p | bd,\ p > C\\ s_{r,p} = \zeta}} p.
$$
We note that $\rho_\zeta \gg r_\zeta^{\frac{1}{2}}$.  One can then re-write \eqref{H_i product bound II} using our above notation to conclude that there exists an integer $\zeta \geq s$ such that 
\begin{align*}
\brac{\frac{ |\mathcal{H}_1|}{H_1^n}}^{2^{r-1}} \brac{\frac{ |\mathcal{H}_2|}{H_2^n}}^{2^{r-2}}\dotsm \brac{\frac{ |\mathcal{H}_r|}{H_r^n}}& \ll P^\eps\ Q^{\sigma-\zeta} (\rho_{\zeta +1} \rho_{\zeta + 2}^2 \dotsm \rho_n^{n - \zeta})^{-1} \\
& \ll P^\eps\ Q^{\sigma-\zeta} \brac{r_{\zeta +1} r_{\zeta + 2}^2 \dotsm r_n^{n - \zeta}}^{-\frac{1}{2}}\\
& = P^\eps\ Q^{\sigma-\zeta} R^{[k]}_\zeta(g, q)^{-\frac{1}{2}}.
\end{align*}

To complete the proof of Lemma \ref{van der corput process} it remains to establish that $g$ satisfies the bound 
$\mathrm{Height}_{P,k}(g) \ll Q^{2- 2^{k+1-d}}$.  Taylor's formula implies that
$$
f_{\vh}(\vx) = f(\vx + \vh) - f(\vx) = \sum_{|\vm|>0} \frac{\vh^\vm}{\vm!} \partial^\vm f(\vx).
$$
Hence there exist forms $G_l$ and $F_l$, each of degree $l$ with $\mathrm{Height}(G_l)\ll 1$ and $\mathrm{Height}(F_l) \ll \mathrm{Height}\brac{f} = O(1)$ such that 
$$
f_{\vh}(\vx) = \sum_{l=0}^{d-1} G_{d-l}(\vh) F_l(\vx).
$$
Supposing that $|\vh| \leq H \leq P$ we have
\begin{align*}
\mathrm{Height}_{P, d-1}\brac{f_\vh} & \ll \sum_{l=0}^{d-1}  P^{l-d+1} |G_{d-l}(\vh)|\mathrm{Height}\brac{F_l} \ll H.
\end{align*}
Hence if $H_1,\dots,H_{r}$ are chosen as in \eqref{eq:choose H}
it follows from  induction that 
$$
\mathrm{Height}_{P, d-r}\brac{f_{\vh_1, \dots, \vh_r}} \ll H_1\dots H_r = Q^{\sum_{i=0}^{r-1} 2^{-i}},
$$
as required.  
\end{proof}

We are now ready to reveal our two main estimates for the exponential sum $S(\alpha,P)$. The first of these is Proposition \ref{birch + van der corput}. It is based on applying van der Corput differencing $d-k$ times (Lemma \ref{van der corput process}) before  applying the Weyl bound to the resulting exponential sum with underlying polynomial of degree at most $k$ (Lemma~\ref{weyl bound}). The second result is Proposition \ref{vdc + cubic lemma}. This is proved using  $d-3$  applications of  van der Corput differencing  (Lemma \ref{van der corput process}) together with an application of
the bound for cubic exponential sums obtained via  Poisson summation 
(Lemma \ref{cubic poisson bound}).

\begin{proposition}[van der Corput + Weyl]\label{birch + van der corput}
Let $
B_k := (k-1)2^{d-1}$ and let $V_k := 2^{d+1 - k} -2$,
for $1 \leq k \leq d$. 
 Suppose that $\mathrm{Height}\brac{f} = O(1)$.  Then 
\begin{align*}
S(\alpha, P) \ll~&
P^{n+\eps}
\Bigbrac{ P^{-2^{1-d}} + \norm{q\alpha}^{\recip{B_k+V_k}} + \bigbrac{qP^{-k}}^\recip{B_k}\\ 
&+ \min\bigset{ q^{-\recip{B_k+V_k}}, \bigbrac{\norm{q\alpha} P^k}^{-\recip{B_k}} }}^{n-\sigma}.
\end{align*}
\end{proposition}

\begin{proof}
Let $1\leq Q\leq P$. 
Applying Lemma \ref{van der corput process} we obtain
$$
\abs{\frac{S(\alpha,P)}{P^n}}^{2^{d-1}} \ll \frac{P^{\eps}}{Q^{(\zeta-\sigma)2^{k-1}}} \abs{\frac{T(\alpha,P)}{P^n}}^{2^{k-1}},
$$
for some exponential sum $T(\alpha,P)$ with underlying polynomial $g$ satisfying the conclusions of the lemma.
In particular $\zeta \geq \max\set{\sigma, \sigma_\infty(g^{[k]})}$.

Setting $\tau := 2 - 2^{k+1-d}$, $\theta := \sigma_\infty(g^{[k]})$ and applying Lemma \ref{weyl bound}, we see that $T(\alpha, P) \ll P^{n+\eps} \Xi^{\frac{n-\theta}{k-1}}$ where
\begin{align*}
\Xi :=  P^{1-k} + \norm{q\alpha} Q^\tau + qP^{-k} + \min\set{Q^\tau q^{-1}, (\norm{q \alpha} P^k)^{-1}}.
\end{align*}
Since $\zeta \geq \theta$, we have $T(\alpha, P) \ll P^{n+\eps} \Xi^{\frac{n-\zeta}{k-1}}$ (this is obvious when $\Xi \leq 1$ and follows from the trivial estimate $T(\alpha, P) \ll P^n$ otherwise).  We thus obtain
\begin{align}
\abs{\frac{S(\alpha,P)}{P^n}}^{2^{d-1}} 
 & \ll \frac{P^{\eps}}{Q^{(n-\sigma)2^{k-1}}} \brac{1+  Q^{(k-1)2^{k-1}} \Xi}^{\frac{n- \zeta}{k-1}} \nonumber\\
  & \leq \frac{P^{\eps}}{Q^{(n-\sigma)2^{k-1}}} \brac{1+  Q^{(k-1)2^{k-1}} \Xi}^{\frac{n- \sigma}{k-1}} \nonumber\\
    &= P^{\eps} \brac{Q^{-{(k-1)2^{k-1}}}+ \Xi}^{\frac{n- \sigma}{k-1}}.\label{Q delta}
\end{align}

Let us take $Q$ such that 
$
Q^{(k-1)2^{k-1}}$ is equal to
$$
\min\set{P^{k-1}, \norm{q\alpha}^{-\frac{ (k-1)2^{k-1}}{\tau + (k-1)2^{k-1}}}, P^k q^{-1}, \max\Bigset{q^{\frac{ (k-1)2^{k-1}}{\tau + (k-1)2^{k-1}}}, \norm{q\alpha} P^k}}.
$$
We may assume that $q \leq P^k$, since the result is  trivial otherwise.  Using this assumption, one can check that $1\leq Q \leq P$, so that our choice of $Q$ is indeed valid.  Moreover, with this choice, the $Q^{-{(k-1)2^{k-1}}}$ term  dominates
in \eqref{Q delta}.  Hence
\begin{align*}
\abs{\frac{S(\alpha,P)}{P^n}}^{2^{d-1}} \ll P^{\eps} Q^{-(n- \sigma)2^{k-1}}
\end{align*}
One can now check that the result follows with the appropriate exponents. 
\end{proof}

\begin{proposition}[van der Corput + cubic Poisson]\label{vdc + cubic lemma} 
Suppose that
$q \leq P^2$,  $\norm{q\alpha} \leq P^{-1}$ and 
 $\mathrm{Height}\brac{f} = O(1)$.    Then
$$
S(\alpha, P) \ll P^{n+\eps}\Bigbrac{\bigbrac{q P^{-2}}^{\recip{2^{d-2}}} + \bigbrac{\norm{q\alpha}P}^{\recip{2^{d-1} - 2}} +  \eta_q^{-1} }^{n-\sigma},
$$
where
\begin{equation}\label{eta defn}
\eta_q := \max\set{ (b^3c_1c_2^2)^{\frac{1}{2^d-2}}, (b^3c_1^2c_2)^{\frac{1}{5\cdot2^{d-2} -4}}}.
\end{equation}
\end{proposition}

\begin{proof}
Proceeding as before, we employ Lemma~\ref{van der corput process} to van der Corput difference down to a cubic exponential sum, which we then estimate with Lemma~\ref{cubic poisson bound}.  We thereby deduce that for any $1\leq Q \leq P$ there exists $\zeta \geq \sigma$ such that 
\begin{align*}
\abs{\frac{S(\alpha, P)}{P^n}}^{2^{d-3}} \ll~& P^\eps\ Q^{-(n- \sigma)} \Big(1+\frac{Q\sqrt{q}}{P} + Q\sqrt{\norm{q\alpha}PQ^\tau} \\
&+ \frac{Q^{1+\frac{\tau}{6}}}{b^{\frac{1}{2}}(c_1c_2)^{\frac{1}{3}}} \min\set{c_1, c_2Q^{\tau}}^{\frac{1}{6}}\Big)^{n-\zeta},
\end{align*}
where
$\tau := 2-2^{4-d}$. Absorbing $Q^{-(n-\sigma)}$ into the brackets we see that the right hand side is at most
$$
P^\eps \brac{Q^{-1} + \frac{\sqrt{q}}{P} + \sqrt{\norm{q\alpha}PQ^\tau} + \frac{Q^{\frac{\tau}{6}}}{b^{\frac{1}{2}}(c_1c_2)^{\frac{1}{3}}} \min\set{c_1, c_2Q^{\tau}}^{\frac{1}{6}}}^{n-\sigma}.
$$
Let us take
$$
Q := \min\set{\frac{P}{\sqrt{q}}, \brac{\norm{q\alpha} P}^{-\recip{2 + \tau}}, \max\set{(b^3 c_1 c_2^2)^{\recip{6 + \tau}}, (b^3 c_1^2 c_2)^{\recip{6 + 2\tau}}}}.
$$
Since $q \leq P^2$ and $\norm{q\alpha} \leq P^{-1}$, we have $1 \leq Q \leq P$.  One can also check that 
$$
Q^{-1} \geq \max\set{\frac{\sqrt{q}}{P}, \sqrt{\norm{q\alpha}PQ^\tau}, \frac{Q^{\frac{\tau}{6}}}{b^{\frac{1}{2}}(c_1c_2)^{\frac{1}{3}}} \min\set{c_1, c_2Q^{\tau}}^{\frac{1}{6}}},
$$
whence the $Q^{-1}$ term dominates and we find that
\begin{align*}
\abs{\frac{S(\alpha, P)}{P^n}}^{2^{d-3}} 
 \ll~& P^\eps\brac{\frac{\sqrt{q}}{P}+\Big(\norm{q\alpha} P}^{\recip{2 + \tau}} \\
 & + \min\set{(b^3 c_1 c_2^2)^{-\recip{6 + \tau}}, (b^3 c_1^2 c_2)^{-\recip{6 + 2\tau}}}\Big)^{n-\sigma}.
\end{align*}
The desired result easily follows.
\end{proof}

\section{The minor arc bound}\label{s:minor}

It follows from the 
definition \eqref{major arc def} of the major arcs that  if  $\alpha \in \m$ then for any $q\in \N$ either
$$
q > P^\Delta \quad \text{or}\quad \norm{q\alpha} > P^{\Delta - d}.
$$
Our objective in this section is to establish the following estimate for the minor arc contribution, which clearly 
suffices for \eqref{eq:minor}.

\begin{lemma}\label{lem:minor}
Suppose that 
$$
n-\sigma
\begin{cases}
> \frac{3}{4} d 2^d - 2d, & \mbox{if $3 \leq d \leq 9$,}\\
\geq \brac{d-\frac{1}{2} \sqrt{d}} 2^d, & \mbox{if $d\geq 10$.}
\end{cases}
$$
Then 
$$
\int_{\m} |S(\alpha, P)| \intd\alpha \ll P^{n-d - \Omega(1)},
$$
where the $\Omega(1)$ term depends at most on $d$ and $\Delta$. 
\end{lemma}

The  work in this section will involve a number of quantities that are defined in terms of $d$ and $k\in \{1,\dots,d\}$ and it is convenient to record them here for ease of reference. 
We put
$$
B_k = (k-1)2^{d-1}, \quad  V_k = 2^{d+1 - k} -2
$$
and 
\begin{equation}\label{eq:xi}
 \xi := \frac{5  - 2^{4-d}}{3 -2^{3-d}}, \quad 
 \gamma := \recip{3\cdot 2^{d-2} - 2},
 \end{equation}
together with 
\begin{equation}\label{A_d defn}
A_d := 7 \cdot 2^{d-4} - \tfrac{5}{4}.
\end{equation}
In particular  $\xi \leq \frac{5}{3}$.
It will also be convenient to define
\begin{equation}\label{C(alpha, q) defn}
C(\alpha, q) :=
 \bigbrac{\norm{q\alpha}P}^{\recip{2^{d-1} - 2}} + \bigbrac{q P^{-2}}^{\recip{2^{d-2}}} + \eta_q^{-1}
\end{equation}
and
\begin{equation}\label{W_k defn}
\begin{split}
W_k(\alpha, q):=~&
P^{-2^{1-d}}+\norm{q\alpha}^{\recip{B_k+V_k}} + \bigbrac{qP^{-k}}^\recip{B_k}\\ 
&+ \min\bigset{ q^{-\recip{B_k+V_k}}, \bigbrac{\norm{q\alpha} P^k}^{-\recip{B_k}} }.
\end{split}
\end{equation}
Next, we define the minimum 
$$
M(\alpha, q): = \min\set{C(\alpha, q), W_3(\alpha, q), \dots, W_d(\alpha, q)}.
$$
In the following result we estimate the minor arc contribution in terms of these quantities.

\begin{lemma}\label{subadditive bound lemma}
There exists $i \in \set{1, 2}$ such that
\begin{align*}
\int_\m  |S(\alpha, P)|\intd\alpha \ll~&  
P^{n-d-\Omega(1)}\\
&+P^{n+ \eps}
\sum_{P^{\Delta_i} \leq q\leq  P^\xi}  \int_{P^{-d_i} \leq \norm{q\alpha} \leq P^{-\xi} }M(\alpha, q)^{n-\sigma} \intd\alpha,
\end{align*}
where
$$
\Delta_1 := \Delta, \quad \Delta_2 := 0, \quad d_1 := d+2, \quad d_2 := d- \Delta.
$$
\end{lemma}

\begin{proof} 
Let
$$
\mathcal{B}(q, Q) := \set{\alpha \in \T : \norm{q\alpha} \leq Q^{-1} \text{ and $\alpha$ is primitive to $q$}}.
$$
Then by Dirichlet's theorem on Diophantine approximation and the definition of the minor arcs, we have 
$$
\m \subset \brac{\bigcup_{P^\Delta \leq q\leq P^{\xi}} \mathcal{B}\bigbrac{q, P^\xi}} \cup \brac{\bigcup_{ q\leq P^\xi} \mathcal{B}\bigbrac{q, P^{\xi}}\setminus \mathcal{B}\bigbrac{q, P^{d-\Delta}} }.
$$
For $\alpha \in \mathcal{B}(q, P^{\xi})$ it follows from 
 Propositions \ref{birch + van der corput} and \ref{vdc + cubic lemma} that 
$$
S(\alpha, P) \ll P^{n+\eps} M(\alpha, q)^{n-\sigma}.
$$
Hence by sub-additivity of integration and dropping the primitivity condition on $q$ and $\alpha$, we deduce that 
\begin{align*}
\int_\m |S(\alpha, P)|\intd\alpha  \ll~& 
P^{n+\eps}
\sum_{P^\Delta \leq q\leq P^{\xi}} \int_{\norm{q\alpha} \leq P^{-\xi}} M(\alpha, q)^{n-\sigma} \intd\alpha\\
& + P^{n+\eps}\sum_{ q\leq P^{\xi}} \int_{P^{\Delta - d}\leq \norm{q\alpha} \leq P^{-\xi} }M(\alpha, q)^{n-\sigma} \intd\alpha, 
\end{align*}
for any $\eps > 0$.
The second term above is of the required form.  In the first term, we need to remove the possibility that $\norm{q\alpha} \leq P^{-d-2}$.  But the  trivial bound $M(\alpha, q) \leq 1$ yields
$$
\sum_{P^\Delta \leq q\leq P^{\xi}} \int_{\norm{q\alpha} \leq P^{-d-2}} M(\alpha, q)^{n-\sigma} \intd\alpha  \ll P^{\xi-d-2}.
$$
Thus the result follows with the $\Omega(1)$ term equal to any positive real strictly less than $\frac{1}{3}$.
\end{proof}

Our aim is to show that when $n-\sigma > \frac{3}{4}d2^d - 2d$ then $M(\alpha, q)^{n-\sigma}$ contributes at most $P^{-d-\Omega(1)}$ once integrated over the range of $(\alpha, q)$ afforded by Lemma~\ref{subadditive bound lemma}. We emphasise that for both $i\in \set{1,2}$ we have
$$
\sum_{P^{\Delta_i} \leq q\leq  P^\xi}  \int_{P^{-d_i} \leq \norm{q\alpha} \leq P^{-\xi} }\intd\alpha \ll 1,
$$
so that a pointwise bound of the form $M(\alpha, q)^{n-\sigma} \ll P^{-d - \Omega(1)}$ suffices for our purposes.  To this end, we first utilise the bound $M(\alpha,q) \leq C(\alpha, q)$, and show that for $\alpha$ and $q$ in the range in question, all terms of $C(\alpha,q)$ are negligible bar possibly the $\eta_q$ term.

\begin{lemma}
Let $\alpha \in \T$ and $q \in \N$ satisfy
$$
q \leq P^{\xi} \quad \text{and} \quad \norm{q\alpha} \leq P^{-\xi},
$$
with $\xi$ as in \eqref{eq:xi}.
Then provided that $n - \sigma > \frac{3}{4} d 2^d - 2d$ we have
\begin{equation}\label{eta bound eqn}
M(\alpha, q)^{n-\sigma} \ll P^{-d -\Omega(1)} +  \eta_q^{-(n-\sigma)}.
\end{equation}
\end{lemma}

 \begin{proof}
 This follows from the inequality $M(\alpha, q) \leq C(\alpha, q)$ and by substituting the bounds on 
 $q$ and $\|q\alpha\|$  into \eqref{C(alpha, q) defn}.
 \end{proof}
 
Next  recall the definition \eqref{eq:xi} of $\gamma$.
When $\eta_q > P^{\gamma}$ and $n-\sigma > \frac{3}{4} d 2^d - 2d$ then  \eqref{eta bound eqn}  implies that $M(\alpha, q)^{n-\sigma} \ll P^{-d - \Omega(1)}$.  It follows that we can restrict our estimation of $M(\alpha, q)$ to those $q$  satisfying $\eta_q \leq P^{\gamma}$.  It is helpful to have an estimate for the number of $q$ which lie in this range.  This is provided by the  following lemma.

\begin{lemma}\label{restricted theta_q estimate}  Define 
Let $R \geq 1$. Then 
$$
\hash\set {q \in \N : \eta_q \leq R} \ll R^{A_d},
$$
where $A_d$ is given by \eqref{A_d defn}.
\end{lemma}

\begin{proof}
Recall the definition of $b, c_1$ and $c_2$ given in \eqref{bcd defn}.  Setting
$$
e := \prod_{p^5 \| q} p,
$$
we have that $e \mid c_2$ and that $c_1/(c_2 e)$ is a squareful positive integer.  Since $\eta_q \geq  (b^3c_1^2c_2)^{\frac{1}{5\cdot2^{d-2} -4}}$, we see that $\eta_q \leq R$ implies  that
$$
\frac{c_1}{c_2 e} \leq \frac{R^{5\cdot 2^{d-3} - 2}}{b^{\frac{3}{2}}c_2^{\frac{3}{2}} e}.
$$ 
As a squareful number is a product of a square and a cube,  the hyperbola method shows that the number of squareful integers less than or equal to $X$ is at most $3X^{1/2}$.  Thus for a fixed choice of $b$ and $c_2$, the number of choices of $c_1$ for which $\eta_q \leq R$ is at most
$$
3\sum_{e \mid c_2} \frac{R^{5\cdot 2^{d-4} - 1}}{b^{\frac{3}{4}}c_2^{\frac{3}{4}} e^{\frac{1}{2}}}.
$$ 
Using the fact that $\eta_q \geq (b^3c_1c_2^2)^{\frac{1}{2^d-2}}$ we see that $\eta_q \leq R$ implies that for fixed $b$ we have 
$
c_2 \leq R^{2^{d-1} - 1}/b^{\frac{3}{2}}.
$
Hence the number of choices for $q$ for which $\eta_q \leq R$ is of order 
\begin{align*}
 \sum_{b} \sum_{c_2 \leq \frac{R^{2^{d-1}}}{Rb^{\frac{3}{2}}}} \sum_{e\mid c_2}  \frac{R^{5\cdot 2^{d-4} - 1}}{b^{\frac{3}{4}}c_2^{\frac{3}{4}} e^{\frac{1}{2}}}
&\ll R^{5\cdot 2^{d-4} - 1}\sum_{b} b^{-\frac{3}{4}} \sum_{e f \leq \frac{R^{2^{d-1}}}{Rb^{\frac{3}{2}}}} f^{-\frac{3}{4}} e^{-\frac{5}{4}}\\
& \ll R^{5\cdot 2^{d-4} - 1}\sum_{b} b^{-\frac{3}{4}} \brac{ \frac{R^{2^{d-1}}}{Rb^{3/2}}}^{\frac{1}{4}}\\
& \ll R^{5\cdot 2^{d-4} - 1+ 2^{d-3} - \frac{1}{4}}.
\end{align*}
This completes the proof of the lemma.
\end{proof}

We are now in a position to use information coming from $W_k(\alpha,q)$ in our definition of $M(\alpha, q)$.  We first show that both the $P^{-2^{1-d}}$ and $qP^{-k}$ terms appearing in $W_k(\alpha,q)$ are negligible when $\eta_q \leq P^{\gamma}$ and $n-\sigma > \frac{3}{4}d2^d - 2d$.  To this end, define
\begin{equation}\label{eq:Wk}
W_k^*(\alpha,q) = \max\set{\norm{q\alpha}^{\recip{B_k+V_k}},\min\bigset{\eta_q^{-1}, (\norm{q\alpha}P^k)^{-\recip{B_k}}}}
\end{equation}
and
$$
M^*(\alpha, q): = \min\set{ W^*_3(\alpha, q), \dots, W^*_d(\alpha, q)}.
$$
Then we have the following result.

\begin{lemma}\label{M to M star lemma} 
Assume that  $n-\sigma > \tfrac{3}{4}d2^d -2d$.
Then there exists  $i\in \set{1, 2}$  such that 
\begin{align*}
\sum_{P^{\Delta_i} \leq q\leq P^{\xi}} 
 &\int_{P^{-d_i} \leq \norm{q\alpha} \leq P^{-\xi} }M(\alpha, q)^{n-\sigma} \intd\alpha \\
& \ll P^{-d-\Omega(1)}
+ \sum_{P^{\gamma_i}\leq \eta_q\leq P^{\gamma} } \int_{P^{-d_i} \leq \norm{q\alpha} \leq P^{-\xi} }
M^*(\alpha, q)^{n-\sigma} \intd\alpha,
\end{align*}
where 
$$
\gamma_i = \frac{\Delta_i}{5\cdot 2^{d-2}-4}.
$$
\end{lemma}

\begin{proof}  If $n-\sigma > \tfrac{3}{4}d2^d -2d$ and $d \geq 3$ then $(n-\sigma)/2^{d-1} > d$, so that
$$
P^{-2^{1-d} (n-\sigma)} \ll P^{-d - \Omega(1)}.
$$
This deals with the $P^{-2^{1-d}}$ term.

For the second term, we note from \eqref{eta defn} that 
$$
\eta_q \geq q^{\frac{1}{5\cdot 2^{d-2}-4}}.
$$
Combining this with \eqref{W_k defn} and \eqref{eta bound eqn} one sees that the lemma follows if for each $3 \leq k \leq d$ we have
$$
\sum_{P^{\gamma_i}\leq \eta_q\leq P^{\gamma}}  P^{-\xi}   (qP^{-k})^{\frac{n-\sigma}{B_k}} \ll P^{-d-\Omega(1)}.
$$
Here we have used the fact that the set $\{\alpha \in \T:\norm{q\alpha} \leq P^{-\xi}\}$ has measure $2P^{-\xi}$.  Using Lemma \ref{restricted theta_q estimate} and the bound $q \leq P^{\xi}$ we have  
\begin{align*}
 P^{-\xi}\sum_{\eta_q \leq P^{\gamma}} (qP^{-k})^{\frac{n-\sigma}{B_k}} & \leq  P^{\xi\brac{\frac{n-\sigma}{B_k}-1}-\frac{k(n-\sigma)}{B_k}}\sum_{\eta_q \leq P^{\gamma}}1\\ & \leq  P^{\xi\brac{\frac{n-\sigma}{B_k}-1}-\frac{k(n-\sigma)}{B_k} + \gamma A_d}.
\end{align*}
One can check from the definitions \eqref{eq:xi} and \eqref{A_d defn} that $\gamma A_d \leq \frac{7}{12}$ and  $\xi \leq \frac{5}{3}$.  Incorporating this together with the fact that $\frac{n-\sigma}{B_k} -1 \geq 0$, we obtain
$$
P^{\xi\brac{\frac{n-\sigma}{B_k}-1}-\frac{k(n-\sigma)}{B_k} + \gamma A_d} \leq P^{-(k-\frac{5}{3})\frac{n-\sigma}{B_k} - \frac{13}{12}}.
$$
To complete the proof of the lemma, we require that 
$$
 \frac{(n-\sigma)(k - \tfrac{5}{3})}{B_k} + \frac{13}{12} > d
$$
for $k \geq 3$. This is equivalent to 
$$
n-\sigma >    \trecip{2}\Bigbrac{\tfrac{k-1}{k-\frac{5}{3}}} \bigbrac{d-\tfrac{13}{12}} 2^{d}=\trecip{2}\Bigbrac{1 + \tfrac{2}{3k - 5}}  \bigbrac{d-\tfrac{13}{12}} 2^{d}.
$$
But $\frac{2}{3k-5} \leq \recip{2}$ (as $k \geq 3$) and the required bound is thus implied by the assumption that $n - \sigma > \tfrac{3}{4} d 2^d - 2d$, together with the estimate $\frac{3\cdot13}{4\cdot12}2^d \geq 2d$, which is valid for $d\geq 3$.
\end{proof}

In view of the preceding result, it remains to analyse the contribution from the term $M^*(\alpha,q)$.
The following result represents a key step in our argument and provides us with a concrete condition under which the bulk of this contribution is satisfactory.

\begin{lemma}\label{inductive argument lemma}  Let $k \geq 3$ and suppose that for each $\ell\in \{k,\dots,d\}$ 
we have 
\begin{equation}\label{reformulated l condition}
n- \sigma > A_d + \tfrac{d}{\ell} B_\ell + (\tfrac{d}{\ell}-1)(B_{\ell+1} + V_{\ell+1}).
\end{equation}
Then
$$
M^*(\alpha,q) \ll E(\alpha,q) +\norm{q\alpha}^{\recip{B_{k}+V_{k}}},
$$
where 
\begin{equation}\label{error estimate}
\sum_{P^{\gamma_i}\leq \eta_q\leq P^{\gamma} } \int_{P^{-d_i} \leq \norm{q\alpha} \leq P^{-\xi} } E(\alpha,q)^{n-\sigma}  \intd\alpha \ll P^{-d - \Omega(1)}.
\end{equation}
\end{lemma}

\begin{proof}
We proceed by inducting downwards on $k$, starting with $k = d$.
By the definition \eqref{eq:Wk} of $W^*_d(\alpha,q)$, it suffices to prove that 
\begin{align*}
\sum_{P^{\gamma_i} \leq \eta_q \leq P^\gamma} \int_{P^{-d_i} \leq \norm{q\alpha} \leq P^{-\xi}} 
\hspace{-0.3cm}
\min\set{\eta_q^{-1}, \bigbrac{\norm{q\alpha}P^d}^{-\frac{1}{B_d}}}^{n-\sigma} \intd\alpha
\ll P^{-d-\Omega(1)}.
\end{align*}
Let us denote by $I_d$ the left hand side of this desired estimate.
Since 
$$
[P^{\gamma_i}, P^{\gamma} ] \times [P^{-d_i} , P^{-\xi}] \subset [1, P] \times [P^{-d-2}, 1],
$$
we can partition the ranges of $\alpha$ and $q$  occurring in $I_d$ into  $O(\log^2 P)$ dyadic rectangles.  Thus, by the pigeon-hole principle, there exists $\delta \gg_{d, \Delta} 1$, $R \leq P^{\gamma}$ and $t \leq P^{-\xi}$ such that 
$$
R \geq P^\delta \quad \mbox{or} \quad t \geq P^{\delta-d}, 
$$
 and
$$
I_d
 \ll (\log P)^2\sum_{R \leq \eta_q \leq 2R} \int_{t \leq \norm{q\alpha} \leq 2t}  \min\set{\eta_q^{-1}, \bigbrac{\norm{q\alpha}P^d}^{-\frac{1}{B_d}}}^{n-\sigma} \intd\alpha.
$$
Using Lemma \ref{restricted theta_q estimate}, we find that
\begin{align*}
I_d&\ll 
(\log P)^2\, 2t\,  \min\set{R^{-1}, \bigbrac{tP^d}^{-\frac{1}{B_d}}}^{n-\sigma}\sum_{R\leq \eta_q \leq 2R} 1\\
& \ll (\log P)^2 t R^{A_d}  \min\set{R^{-1}, \bigbrac{tP^d}^{-\frac{1}{B_d}}}^{n-\sigma}.
\end{align*}
Now 
for any $a, b, X, Y \geq 0$ we have the inequality
$
\min\set{X,Y}^{a+b} \leq X^aY^b.
$
Hence to establish the base case, it suffices to find $a,b \geq 0$ with $a+b = n-\sigma$ such that 
$$
R^{A_d -a}t^{1- \frac{b}{B_d}} P^{-\frac{db}{B_d}} \ll P^{-d-\Omega(1)}.
$$
Notice that when $\ell = d$, condition \eqref{reformulated l condition} reduces to 
$
n-\sigma > A_d + B_d.
$
There are two cases to consider.
Suppose first that $R \geq P^{\delta}$.   In this case we can take $a = A_d + \Omega(1)$ and $b = B_d$, which gives
$$
R^{A_d -a}t^{1- \frac{b}{B_d}} P^{-\frac{db}{B_d}} \leq P^{-\Omega(\delta) -d}. 
$$
Alternatively, we  suppose that $t \geq P^{\delta-d}$. But in this case we take $a =A_d $ and $b = B_d + \Omega(1)$, giving 
$$
R^{A_d -a}t^{1- \frac{b}{B_d}} P^{-\frac{db}{B_d}} \leq 
P^{-\Omega(\frac{\delta}{B_d}) - d}.
$$
The base case then follows.

Next we turn to the induction step.  Let us assume that $3 \leq k < d$.  By the induction hypothesis  we have
$$
M^*(\alpha,q) \ll E(\alpha,q) +\norm{q\alpha}^{\recip{B_{k+1}+V_{k+1}}},
$$  
where $E(\alpha, q)$ satisfies \eqref{error estimate}.  Combining this with the definition of $W^*_k(\alpha,q)$, it suffices to prove that
$$
\sum_{P^{\gamma_i} \leq \eta_q \leq P^\gamma} \int_{P^{-d_i} \leq \norm{q\alpha} \leq P^{-\xi}} 
\hspace{-0.5cm}
\min\set{\eta_q^{-1}, \bigbrac{\norm{q\alpha}P^k}^{-\frac{1}{B_k}}, \norm{q\alpha}^{\recip{B_{k+1} + V_{k+1}}}}^{n-\sigma} \intd\alpha
$$
is $O( P^{-d-\Omega(1)})$.
As before, breaking into dyadic rectangles, we see that it suffices to find $a,b,c \geq 0$ with $a+b+c = n-\sigma$ such that for $R \leq P^\gamma$ and $t \leq P^{-\xi}$ we have
$$
R^{A_d - a} t^{1 - \frac{b}{B_k} + \frac{c}{B_{k+1} + V_{k+1}}}P^{-\frac{kb}{B_k}} \ll P^{-d - \Omega(1)}.
$$
Recall our assumption \eqref{reformulated l condition}.  Taking 
$$
a = A_d, \quad  b = \tfrac{dB_k}{k}, \quad
 c = (\tfrac{d}{k}-1)(B_{k+1} + V_{k+1}) + \Omega_d(1),
$$
one readily verifies  that the desired estimate holds.
\end{proof}

The size of the term $\|q\alpha\|^{\frac{1}{B_k+V_k}}$ in Lemma 
\ref{inductive argument lemma} is smallest when $k$ is minimal.
When the condition \eqref{reformulated l condition} holds for every $3\leq \ell\leq d$, the following result shows that 
 the contribution from the term $\|q\alpha\|^{\frac{1}{B_3+V_3}}$ is satisfactory under the assumptions of Lemma \ref{lem:minor}.

\begin{lemma}\label{k=3 lemma}  Suppose that $n-\sigma > \frac{3}{4} d2^d - 2d$ and $d \geq 3$.  Then 
$$
\sum_{P^{\gamma_i}\leq \eta_q\leq P^{\gamma} } \int_{P^{-d_i} \leq \norm{q\alpha} \leq P^{-\xi} } \min\set{\eta_q^{-1}, \norm{q\alpha}^{\recip{B_3 + V_3}}}^{n-\sigma}  \intd\alpha \ll P^{-d - \Omega(1)}.
$$
\end{lemma}

\begin{proof}
An easy calculation reveals that 
$$
\tfrac{3}{4} d2^d - 2d \geq 
 A_d + (\tfrac{d}{\xi} - 1)(B_3 + V_3),
 $$
 for $d\geq 3$. Hence 
\begin{equation}\label{k=3 condition}
\begin{split}
n- \sigma & >   A_d + (\tfrac{d}{\xi} - 1)(B_3 + V_3),
\end{split}
\end{equation}
under the assumptions of the lemma. 
As before, we split into dyadic intervals and deduce that it suffices to find $a + b = n-\sigma$ with 
$$
R^{A_d - a} t^{1+ \frac{b}{B_3 + V_3}} \ll P^{-d - \Omega(1)},
$$
where $R \leq P^{\gamma}$ and $t \leq P^{-\xi}$.  Taking $a = A_d$, the inequality \eqref{k=3 condition} ensures that
$$
b = (\tfrac{d}{\xi} - 1)(B_3 + V_3) + \Omega_d(1).
$$
Hence
$$
R^{A_d - a} t^{1+ \frac{b}{B_3 + V_3}} \leq P^{-\xi(1+ \frac{b}{B_3 + V_3})} \leq P^{-d - \Omega_d(1)},
$$
as required.
\end{proof}

We now have everything in place to establish Lemma \ref{lem:minor}. 
Suppose first that $3\leq d\leq 9$ and 
$n-\sigma>\tfrac{3}{4} d 2^d - 2d$.
Combining Lemma \ref{subadditive bound lemma} and Lemma \ref{M to M star lemma}, there exists $i \in \set{1,2}$ such that
\begin{align*}
\int_{\m}|S(\alpha, P)| \intd\alpha \ll  
P^{n-d - \Omega(1)} + P^{n+\eps} 
\hspace{-0.3cm}
\sum_{P^{\gamma_i}\leq \eta_q\leq P^{\gamma} } \int_{P^{-d_i} \leq \norm{q\alpha} \leq P^{-\xi} }
\hspace{-0.3cm}
M^*(\alpha, q)^{n-\sigma} \intd\alpha
\end{align*}
One can check (most expediently by computer) that when $3 \leq d \leq 9$ and $3 \leq k \leq d$ we have
$$
\tfrac{3}{4} d 2^d - 2d + 1 > A_d + \tfrac{d}{k} B_k + (\tfrac{d}{k}-1)(B_{k+1} + V_{k+1}).
$$
The hypotheses of Lemma \ref{inductive argument lemma} are therefore satisfied with $k= 3$, so that
\begin{align*}
 \sum_{P^{\gamma_i}\leq \eta_q\leq P^{\gamma} } &\int_{P^{-d_i} \leq \norm{q\alpha} \leq P^{-\xi} } M^*(\alpha, q)^{n-\sigma} \intd\alpha\\& \ll \sum_{P^{\gamma_i}\leq \eta_q\leq P^{\gamma} } \int_{P^{-d_i} \leq \norm{q\alpha} \leq P^{-\xi} } \min\set{\eta_q^{-1}, \norm{q\alpha}^{\recip{B_3 + V_3}}}^{n-\sigma}  \intd\alpha.
\end{align*}
But the latter quantity is $O(P^{-d- \Omega(1)})$ by Lemma \ref{k=3 lemma}, which therefore concludes the proof of  Lemma \ref{lem:minor} when $3\leq d\leq 9$.

Suppose now that $d \geq 10$  and 
$n - \sigma \geq ({d-\frac{1}{2} \sqrt{d}}) 2^d$.
As above it  suffices to show that for all $k$ in the range $3 \leq k \leq d$ we have
$$
d 2^d - \sqrt{d}\ 2^{d-1} > A_d + \tfrac{d}{k} B_k + (\tfrac{d}{k}-1)(B_{k+1} + V_{k+1}).
$$
Substituting the definitions of $A_d, B_k$ and $V_k$ into the right-hand side, we see that it is strictly less than
\begin{align*}
d 2^d - \brac{ \tfrac{d}{2}\brac{\trecip{k} - \trecip{k2^{k-1}}} + \trecip{2}\brac{k-\tfrac{7}{8}}}2^d & <d 2^d - \Bigbrac{ \tfrac{d}{2}\brac{\trecip{2k}} + \trecip{2}\brac{\tfrac{k}{2}}}2^d\\
& = d2^d - \tfrac{\sqrt{d}}{4} \brac{\tfrac{\sqrt{d}}{k} + \tfrac{k}{\sqrt{d}}}2^d\\
& \leq d2^d - \sqrt{d}\cdot 2^{d-1}.
\end{align*}
This therefore concludes the proof of Lemma \ref{lem:minor}.

\begin{remark}\label{remark}
It is somewhat disappointing that we are unable to do better when $n\geq 10$. However, one easily confirms that
the maximum of the conditions  \eqref{reformulated l condition} is asymptotically at least $d 2^d$. Thus one cannot obtain the required minor arc bound with roughly  $\frac{3}{4} d 2^d$ variables.  
\end{remark}	

\section{The major arc asymptotic}\label{major arc section}

The purpose of this section is to establish \eqref{eq:major} under suitable hypotheses on the form $F$ and on the parameter $\Delta$ occurring in the definition \eqref{major arc def} of $\major$.    

According to our local solubility hypothesis in Theorems \ref{general d improvement theorem} and 
\ref{small d improvement theorem} 
we may assume that the system 
\eqref{eq'} has a solution over the reals. Thus  
there exists a vector
 $\x_0\in \RR^n$ such that $F(\x_0)=0$ and
$\nabla F (\x_0)\neq \mathbf{0}$. This vector is to be considered fixed once and for all in what follows. 
It will be convenient to work with a weight function 
that forces us to count points lying very close to $\x_0$. 
For any $\delta\in (0,1]$, we define the function $\omega: \RR^n 
\rightarrow 
\RR_{\geq 0}$ by 
$$
\omega(\x):=w\left(\delta^{-1}\|\x-\x_0\|_2\right),
$$
where $\|\y\|_2=\sqrt{y_1^2+\cdots+y_n^2}$ and 
$$
w(x):=\begin{cases}
e^{-\frac{1}{1-x^2}}, & \mbox{if $|x|<1$},\\
0, & \mbox{if $|x|\geq 1$}.
\end{cases}
$$
We will require  $\delta$ to be 
sufficiently small, with  $1\ll\delta \leq 1$.
It is clear that $\omega$ belongs to the class $ \mathcal{S}^+(\vc)$
from Definition \ref{smooth}
 for a suitable 
infinite tuple $\vc=(c,c_0,c_1,\dots)$ depending on $\delta$ and $\x_0$.

Let us define 
$$
S_q(a):=\sum_{\y\bmod{q}}e_q\big(aF(\y)\big),
$$
for $a\in \ZZ$ such that $(a,q)=1$, together with the associated 
truncated singular series
\begin{equation}\label{21-sing}
\ss(R)=\sum_{q\leq R} \frac{1}{q^n}\sum_{\substack{a \bmod{q}\\ 
(a,q)=1}}
S_q(a),
\end{equation}
for any $R>1$. 
We put  $\ss=\lim_{R\rightarrow \infty}\ss(R)$, whenever this limit exists.
Next, let
\begin{equation}
  \label{21-si}
\mathfrak{I}(R)=\int_{-R}^{R}  
\int_{\RR^n} \omega(\x)
e\big(\gamma F(\x)\big)\d\x\d \gamma,
\end{equation}
for any $R>0$. We put  $\mathfrak{I}=\lim_{R\rightarrow \infty}
\mathfrak{I}(R)$, whenever the limit  exists.  
The main aim of this section is to establish the
following result.

\begin{lemma}\label{lem:major}
Assume that $n-\sigma> \frac{3}{4}(d-1)2^{d}$.   Then the singular series 
$\ss$ and the singular integral $\mathfrak{I}$ are absolutely 
convergent.  Moreover, if we choose
$\Delta=\frac{1}{6}$
then
$$
\int_{\major}S(\alpha)\d\alpha=\ss \mathfrak{I}P^{n-d}
+O(P^{n-d-\Omega(1)}).
$$
\end{lemma}

Here the leading constant 
$\ss \mathfrak{I}$ is a product of local densities and, in the usual way, one has 
$\ss \mathfrak{I}>0$ under the local solubility assumptions of 
 Theorems 
 \ref{general d improvement theorem} and 
\ref{small d improvement theorem}.
Once taken in conjunction with Lemma \ref{lem:minor}, the 
 proof of our main results   will therefore stand once Lemma \ref{lem:major} is verified. 

\medskip

Our treatment of Lemma \ref{lem:major} is standard and  closely follows the argument of Birch 
\cite[\S 5]{birch}, as revisited in  \cite[\S 10]{41}.
Thus we shall allow ourselves to be brief.
For $q\leq P^{\Delta}$ and  $a\in \ZZ$ coprime to $q$, 
let us put $\al=a/q+\theta$ for any 
$\alpha\in  \major_{a,q}$.
To begin with, the argument of 
\cite[Lemma~5.1]{birch} (cf. \cite[Eq.~(10.5)]{41}) 
easily gives
\begin{equation}
  \label{eq:train1}
S(\alpha, P)=q^{-n}P^n S_q(a) I(\theta P^d) +O\left(
q|\theta|P^{n+d-1}+qP^{n-1}\right),
\end{equation}
where $S_q(a)$ is given above and we put
$$
I(\gamma)=
\int_{\RR^n} \omega(\x)
e\big(\gamma F(\x)\big)\d\x
$$
for any $\gamma\in \RR$.
Recalling that $|\theta|\leq P^{-d+\Delta}$ and $q\leq P^\Delta$ on the major arcs, this  implies that 
$$
S(\alpha, P)=q^{-n}P^n S_q(a) I(\theta P^d) +O(P^{n-1+2\Delta}).
$$
Noting that  the major arcs have measure
$O(P^{-d+3\Delta})$, it  now follows that
\begin{equation}\label{21-maj}
\int_{\major} S(\alpha, P)\d\al=
P^{n-d}\mathfrak{S}(P^\Delta)\mathfrak{I}(P^\Delta) +O(
P^{n-d-1+5\Delta}), 
\end{equation}
where $\mathfrak{S}(P^\Delta)$ is given by \eqref{21-sing}, and
$\mathfrak{I}(P^\Delta)$ is given by \eqref{21-si}.

Next we claim that 
\begin{equation}\label{eq:integral-lemma}
I(\gamma)\ll \min\set{1,|\gamma|^{-\frac{n-\sigma}{(d-1)2^{d-1}}+\eps} }.
\end{equation}
The argument for this is based on  \cite[Lemma~5.2]{birch}
(cf. \cite[Lemma~24]{41}).
The estimate $I(\gamma)\ll 1$ is trivial. In proving the second
estimate we may clearly assume that $|\gamma|>1$. 
Taking $a=0$ and $q=1$ in \eqref{eq:train1}, we deduce that
$$
S(\alpha, P)=P^n I(\al P^d) +O\big((|\al| P^d+1)P^{n-1}\big),
$$
for any $P\geq 1$.  On the other hand, assuming that $|\al|<P^{-\frac{d}{2}}$,
Lemma \ref{weyl bound} gives
$$
S(\alpha, P) \ll P^{n+\eps} (|\alpha|P^{d})^{-\frac{n-\sigma}{(d-1)2^{d-1}}}.
$$
Writing $\al P^d=\gamma$, we may combine these estimates to obtain
$$
I(\gamma)\ll 
|\gamma|
^{-\frac{n-\sigma}{(d-1)2^{d-1}}}P^{\eps}
+|\gamma| P^{-1},
$$
when $|\gamma|<P^{\frac{d}{2}}$.  Finally we observe that $I(\gamma)$ is
independent of $P$.  Thus we are free to choose
$P=|\gamma|^{1-\frac{n-\sigma}{(d-1)2^{d-1}}}$, which thereby establishes  \eqref{eq:integral-lemma}.

Suppose that $n-\sigma> \frac{1}{2}(d-1)2^{d}$.
It now follows from \eqref{eq:integral-lemma} that
\begin{align*}
\mathfrak{I}-\mathfrak{I}(R)
=\int_{|\gamma|\geq R} I(\gamma) \d\gamma
&\ll
\int_{R}^{\infty} \min\{1,\gamma^{-\frac{n-\sigma}{(d-1)2^{d-1}}+\eps}\} \d\gamma\\
&\ll R^{1
-\frac{n-\sigma}{(d-1)2^{d-1}}+\eps}.
\end{align*}
This shows that $\mathfrak{I}$ is absolutely convergent
for 
$n-\sigma> \frac{1}{2}(d-1)2^{d}$, which is more than enough for  Lemma \ref{lem:major}.

Next we need to show that 
\begin{equation}\label{eq:27-record}
\mathfrak{S}-\mathfrak{S}(R)\ll R^{-\eta},
\end{equation}
for some $\eta>0$, provided that 
for $n-\sigma>\frac{3}{4}(d-1)2^{d}$.  
Assuming this to be the case for the moment and 
observing that 
$\mathfrak{I}(P^\Delta)\ll 1$, 
it follows from \eqref{21-maj} that 
\begin{align*}
\int_{\mathfrak{M}} S(\alpha, P)\d\al
=~&
\ss P^{n-d}\mathfrak{I}(P^\Delta)
+O\big(P^{n-d-1+5\Delta}+P^{n-d-\Delta\eta}\big)\\
=~&
\ss \mathfrak{I}P^{n-d}\\
&+O\big(P^{n-d-1+5\Delta}+P^{n-d-\Delta\eta}+P^{n-d-\Delta
(\frac{n-\sigma}{(d-1)2^{d-1}}-1) +\eps}
\big).
\end{align*}
We therefore obtain the statement of Lemma \ref{lem:major} by choosing
$\Delta=\frac{1}{6}$ and taking $\eps>0$ to be
sufficiently small.

Turning finally to 
the proof of \eqref{eq:27-record}, 
we put 
$$
A(q)=q^{-n}\sum_{\substack{a
      \bmod{q}\\ (a,q)=1}}S_q(a).
$$
This is a multiplicative function of $q$. 
It follows from  Birch \cite[Lemma~5.4]{birch}
that
\begin{equation}
  \label{eq:ap-1}
  A(p^k)\ll p^{k(1-\frac{n-\sigma}{(d-1)2^{d-1}})+\eps}.
\end{equation}
When $k=1$ we instead call upon the estimate
$
A(p)\ll p^{1-\frac{n-\sigma}{2}}.
$ 
This is  established by induction on $\sigma$, the
inductive base $\sigma=0$ being taken care of by Deligne's 
estimate \cite{deligne}.  The general case
is reduced to this situation by appropriate hyperplane sections.

We now 
establish \eqref{eq:27-record} under the assumption that $n-\sigma>\frac{3}{4}(d-1)2^{d}$. 
Let us write $q=uv$, where $u$ is the square-free part of $q$. 
Then, by multiplicativity, we have
$
A(u)\ll u^{1-\frac{n-\sigma}{2}+\eps}.
$
Once combined with \eqref{eq:ap-1}, this 
shows there exists $\eta>0$ such that 
\begin{align*}
\big|\ss-\ss(R)\big| 
&\ll \sum_{q=uv>R} u^{1-\frac{n-\sigma}{2}+\eps}\
v^{1-\frac{n-\sigma}{(d-1)2^{d-1}}+\eps}\\
&\ll \sum_{q=uv>R} u^{-2-\eta} v^{-\frac{1}{2}-2\eta}\\
&\ll R^{-\eta}\sum_{u,v=1}^{\infty}u^{-2} v^{-\frac{1}{2}-\eta},
\end{align*}
provided that $\eps>0$ is taken to be sufficiently small.
Since  the number of square-full integers $v\in (V,2V]$ is $O(V^{1/2})$, 
the sum over $v$ is seen to be convergent, as is the sum over $u$.
The bound recorded in 
\eqref{eq:27-record} therefore  follows.

\end{document}